\documentclass[12pt,leqno]{article}

\usepackage{geometry}

\usepackage{graphicx}
\usepackage{epstopdf,epsfig}

\usepackage[font=small]{caption}
\usepackage{amsmath,amsthm,latexsym,amscd,esint,enumerate}
\usepackage{amsfonts,amssymb}
\usepackage{MnSymbol}
 
\usepackage{xy}
\xyoption{all}

\usepackage[explicit]{titlesec} 
\titleformat{\section}{\large\scshape\raggedright}{\thesection.}{0.5em}{#1}[\titlerule]
\titlespacing{\section}{0pt}{5pt}{5pt} 
\titleformat{\subsection}[runin]{\scshape}{}{0em}{\underline{#1}}
\titlespacing{\subsection}{0pt}{0pt}{1em}

\swapnumbers

\theoremstyle{plain}
    \newtheorem{theorem}{Theorem}[section]
    \newtheorem{lemma}[theorem]{Lemma}
    \newtheorem*{lemma*}{Lemma}
    \newtheorem*{proposition*}{Proposition}
    \newtheorem*{theorem*}{Theorem}
    \newtheorem{corollary}[theorem]{Corollary}
    \newtheorem{proposition}[theorem]{Proposition}
    
\theoremstyle{definition}
    \newtheorem{definition}[theorem]{Definition}
    \newtheorem{example}[theorem]{Example}
    \newtheorem{examples}[theorem]{Examples}

    \newtheorem{remark}[theorem]{Remark}
    \newtheorem{remarks}[theorem]{Remarks}

\theoremstyle{remark}

\numberwithin{equation}{section}

\newcounter{actr}
\newenvironment{alist}{\begin{list}
                         {\textup{(\alph{actr})}}
                         {\usecounter{actr}
                          \setlength{\topsep}{0.0truein}
                                                    \setlength{\itemsep}{-3pt}
                                                  \setlength{\labelwidth}{0.3truein}}}
                      {\end{list}}

\newcommand{\R}{\mathbb{R}}
\newcommand{\C}{\mathbb{C}}

\renewcommand{\H}{\mathcal{H}}

\newcommand{\dps}{d}

\newcommand{\deltaps}{\delta}
\newcommand{\Deltaps}{\Delta}

\newcommand{\weight}{w}

\newcommand{\cat}{{\textnormal{CAT}}}

\begin{document}   

\title{\large\scshape A Differential Complex for CAT(0) Cubical Spaces}  
\date{}

\author{\scshape J.~Brodzki%
\footnote{J.~B.~was supported in part by EPSRC grants EP/I016945/1 and EP/N014189/1.},  
E.~Guentner%
\footnote{E.~G.~was supported in part by a grant from the Simons Foundation (\#245398).}, 
N.~Higson%
\footnote{N.~H.~was supported in part by NSF grant DMS-1101382.}%
}
\AtEndDocument{\bigskip{\footnotesize%
  \textsc{Department of Mathematical Sciences, University of Southampton, Southampton SO17 1BJ, UK}\par
  \textit{E-mail address}:  \texttt{j.brodzki@soton.ac.uk} \par
\addvspace{\medskipamount}
  \textsc{Department of Mathematical Sciences, University of Hawaii at Manoa, Honolulu, HI 96822, USA} \par
  \textit{E-mail address}:  \texttt{erik@math.hawaii.edu} \par
\addvspace{\medskipamount}
  \textsc{Department of Mathematics, Penn State University, University Park, PA 16802, USA} \par
  \textit{E-mail address}:  \texttt{higson@psu.edu}
}}

\maketitle

\section{Introduction}

In the 1980's Pierre Julg and Alain Valette \cite{JulgValetteCRAS,JulgValetteQ_p}, and also Tadeusz Pytlik and  Ryszard Szwarc \cite{PytlikSzwarc}, constructed and studied a certain   Fredholm operator  associated to a simplicial tree.   The operator  can be defined  in at least two ways: from a combinatorial flow on  the tree, similar  to  the flows  in Forman's discrete Morse theory \cite{Forman98}, or from  the theory of unitary operator-valued cocycles \cite{Pimsner87,Valette90}.  There are applications of the theory surrounding the operator to 
$C^*$-algebra K-theory \cite{JulgValetteCRAS,JulgValetteQ_p}, to the theory of completely bounded representations of groups that act on trees \cite{PytlikSzwarc}, and to the Selberg principle in the representation
theory of $p$-adic groups \cite{JulgValetteSelberg1,JulgValetteSelberg2}.

The crucial property of the Fredholm operator   introduced by Julg and Valette is that it is the initial operator in a    \emph{continuous family  of Fredholm operators}  parametrized by a closed interval.   The applications all emerge from the properties of the family in the circumstance where a group $G$ acts properly on the underlying tree, in which case all the operators in the  family act  on Hilbert spaces that carry unitary representations of $G$. Roughly speaking, the family connects the  regular representation   of $G$  to the trivial representation within an index-theoretic context.

This calls to mind Kazhdan's property T \cite{Kazhdan67,BekkaHarpeValette}, or rather the negation of property T, as well as Haagerup's property \cite{Haagerup78,CherixEtAl}, which is a strong negation of property T.  Groups that act on trees are known to have the Haagerup property (this is essentially due to Haagerup himself), and the Julg-Valette, Pytlik-Szwarc construction is perhaps best viewed as a geometric incarnation of this fact.  An immediate consequence is the $K$-theoretic amenability of any group that acts property on a tree \cite{Cuntz,JulgValetteQ_p}, which is another strong negation of property T. 

The main aim of this paper is to extend the constructions of Julg and Valette, and Pytlik and Szwarc,  to  $\cat(0)$ cubical spaces (a one-dimensional $\cat(0)$ cubical space is the same thing as a  simplicial tree).  A secondary aim is to  illustrate the utility of the extended construction by   developing an application to operator $K$-theory and  giving a new proof of  $K$-amenability for groups that act properly on bounded-geometry $\cat(0)$-cubical spaces.  But we expect there will be other uses for our constructions, beyond operator $K$-theory.

 We shall associate to each  bounded geometry  $\cat(0)$ cubical space not a Fredholm operator but a \emph{differential complex} with finite-dimensional cohomology.  The construction is rather more challenging for general $\cat(0)$ cubical spaces than it is for trees. Whereas for trees there is a more or less canonical notion of \emph{flow} towards a distinguished base vertex in the tree, in higher dimensions this is not so, and for example a vertex is typically connected to a given base vertex by a large number of edge-paths.  In addition, the need to consider higher-dimensional cubes, and the need  to impose  the  condition $d^2 =0$, oblige us to carefully consider orientations of cubes in a way that is quite unnecessary for trees. 
 
 More interesting still is the problem of defining the final complex in the one-parameter family of complexes that we aim to construct.     To solve it we shall rely on the theory of \emph{hyperplanes} in $\cat(0)$ cubical spaces \cite{NibloReevesNormal}.   In the case of a tree the hyperplanes are simply the midpoints of edges, but in general they have a nontrivial geometry all of their own; in fact they are $\cat(0)$ cubical spaces in their own rights.  
 
 We shall also introduce and study a related notion of \emph{parallelism} among the cubes in a $\cat(0)$ cubical space.  In a tree, any two vertices are parallel, while no two distinct edges are parallel, but in higher dimensions  parallelism  is more subtle.   For instance in a finite tree  the number of vertices is precisely one plus the number of edges (this simple geometric fact is in fact an essential part of the Julg-Valette, Pytlik-Szwarc construction).   But the proof of the following generalization to higher dimensions  is quite a bit more involved.

\begin{proposition*} 
If X is finite $\cat(0)$ cubical space, then the number of vertices of $X$ is equal to the number of parallelism classes of cubes of all dimensions.
\end{proposition*}

 We expect  that parallelism and the other aspects of our constructions, will be of interest and value elsewhere in the theory  of $\cat(0)$ cube complexes.

One last  challenge comes in passing from $\cat(0)$ cubical geometry to Fredholm complexes and operator $K$-theory.  There are two standard paradigms in operator $K$-theory,   of \emph{bounded} cycles and \emph{unbounded} cycles, but the geometry we are faced with here forces us to consider a hybrid of the two. However once this is done we shall arrive at our application:

\begin{theorem*}
If a second countable and locally compact group   $G$ admits a proper action on a bounded geometry $\cat(0)$ cube complex, then $G$ is $K$-amenable. 
\end{theorem*}

Groups that act properly on $\cat(0)$ cube complexes are known to have the Haagerup property \cite{NibloRoller}, and they were proved to be $K$-amenable in \cite[Theorem 9.4]{HigsonKasparov}.  The advantage of the present approach is that the constructions in the proof are all tied to the finite-dimensional cube complex itself, whereas in \cite{HigsonKasparov} the authors rely on an auxilliary action of the group on an infinite-dimensional Euclidean space that is rather hard to understand directly.  

Here is a brief outline of the paper.  After reviewing the concept of hyperplane in Section 2 we shall study orientations and define our Julg-Valette complex in Section 3.  We shall introduce parallelism in Section 4 and define the final complex (we shall call it the Pytlik-Szwarc complex) in Section 5.  The one-parameter family of complexes connecting the two will be constructed in stages, in Sections 6, 7 and 8, and the application to operator $K$-theory will be the subject of Sections 9 and 10.
 

\section{Cubes and Hyperplanes}
\label{sec-cubes-etc}

We shall begin by  fixing some basic notation concerning the cubes and
hyperplanes in a $\cat(0)$ cube complex. We shall follow the
exposition of Niblo and Reeves in \cite{NibloReevesNormal},   with
some adaptations.  

Throughout the paper $X$ will denote a $\cat(0)$ cube complex as in
\cite[Section 2.2]{NibloReevesNormal}.  Though not everywhere
necessary, we shall assume throughout that $X$ is
finite-dimensional, and that it has \emph{bounded geometry} in
the sense that the number of cubes intersecting any one cube $C$ is
uniformly bounded as $C$ varies over all cubes.

Every  $q$-cube contains exactly  $2q$  codimension-one faces.  Each
such face is disjoint from precisely one other, which we shall call
the \emph{opposite} face.

We shall use the standard terms \emph{vertex} and \emph{edge} for
$0$-dimensional and $1$-dimensional cubes. 

The concept of a \emph{midplane} of a cube is introduced in
\cite[Section 2.3]{NibloReevesNormal}.  If we identify a $q$-cube with
the standard cube $[-\frac 12 , \frac 12 ]^q$ in $\R^q$, then the
midplanes are precisely the intersections of the cube with the
coordinate hyperplanes in $\R^q$ (thus the midplanes of a cube $C$ are
in particular closed subsets of $C$). A $q$-cube contains precisely
$q$ midplanes (and in particular a vertex contains no midplanes)

Niblo and Reeves describe an equivalence relation on the set of all
midplanes  in a cube complex: two midplanes
are \emph{\textup{(}hyperplane\textup{)} equivalent} if they can be arranged as the first
and last members of a finite sequence of midplanes for which the
intersection of any two consecutive midplanes is again a midplane.  

\begin{definition}\label{def-hyperplane}\textup{(See \cite[Definition 2.5]{NibloReevesNormal}.)}
A \emph{hyperplane} in $X$ is the union of the set of all midplanes in
an equivalence class of midplanes.  A hyperplane \emph{cuts}
a cube if it contains a midplane of that cube.  When a hyperplane cuts
an edge, we say that the edge \emph{crosses} the hyperplane.  See Figure~\ref{fig:hyperplane}.   

\end{definition}

\begin{figure}[ht] 
   \centering
   \includegraphics[scale=0.5]{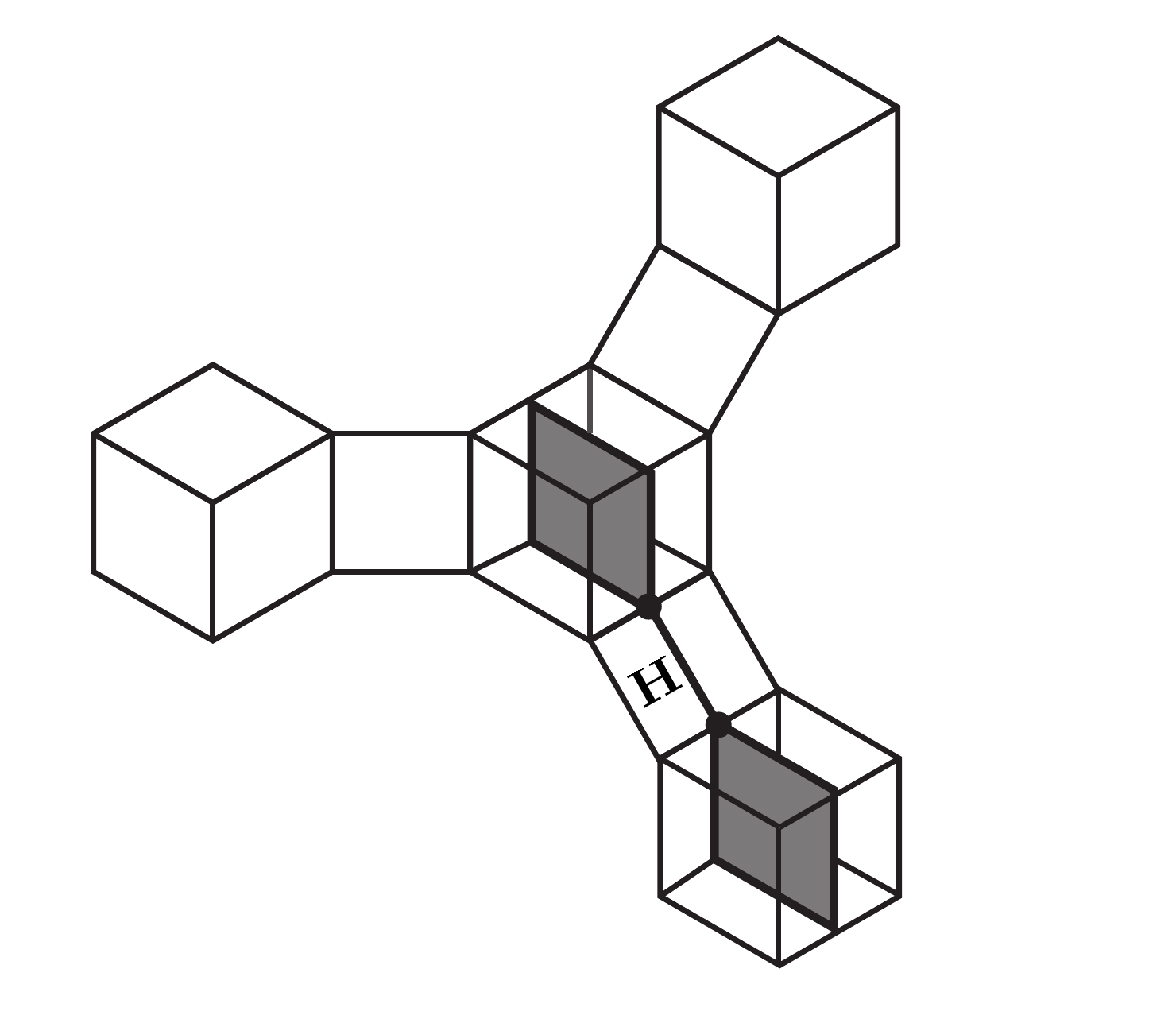} 
    \caption{The  hyperplane $H$ is the union of three midplanes.}
  \label{fig:hyperplane}
\end{figure}

\begin{examples}
If $X$ is a tree, then the hyperplanes are precisely the midpoints of
edges.  If $X$ is the plane, divided into cubes by the integer
coordinate lines, then hyperplanes are the half-integer coordinate
lines. 
\end{examples}

Hyperplanes are particularly relevant in the context of $\cat(0)$
cube complexes (such as the previous two examples) for the following
reason:

\begin{lemma}\textup{(See \cite[Theorem 4.10]{SageevEndsOfGroups} or \cite[Lemma 2.7]{NibloReevesNormal}.)}
If $X$ is a $\cat(0)$ cube complex, then every hyperplane is a totally
geodesic subspace of $X$ that separates $X$ into two connected
components. \qed 
\end{lemma}

The components of the complement of a hyperplane are the two
\emph{half-spaces} associated to the hyperplane.  The half-spaces are
open, totally geodesic subsets of $X$. Moreover the union of all cubes
contained in a given half-space is a $\cat(0)$ cube complex in its own
right, and a totally geodesic subcomplex of $X$.

Later on, it will be helpful to approximate an infinite complex by finite complexes, as follows. 

\begin{lemma}
\label{lem-increasing-union}
Every bounded geometry $\cat(0)$ cube complex $X$ is an increasing
union of \emph{finite}, totally geodesic  $\cat(0)$ subcomplexes $X_n$
whose hyperplanes  are precisely the nonempty  intersections of the
hyperplanes in $X$ with $X_n$. 
\end{lemma}

\begin{proof}
Fix a base point   in $X$ and an integer   $n> 0$.   Form the set of
all hyperplanes whose distance  to the base point is $n$ or
greater, and then form the intersection of all the half-spaces for
these hyperplanes that contain the base point . Denote by $X_n$ the
union of all cubes that are included in this intersection; it is a
totally geodesic subset of $X$ and so a $\cat(0)$ cube
complex. Moreover the intersection of any hyperplane in $X$ with $X_n$ is
connected.  The union of all the $X_n$ as $n\to \infty$ is $X$ and,
since the set of hyperplanes of distance \emph{less} than $n$ to
the base point is finite, each $X_n$ is a finite subcomplex of $X$.  
\end{proof} 

\begin{definition}\label{def-vertex-adjacent}
A hyperplane and a vertex are \emph{adjacent} if the vertex is included in an edge that
crosses the hyperplane. 
\end{definition}

\begin{lemma}
\label{lem-3-hyperplanes}
If $k$  hyperplanes in a $\cat(0)$ cube complex intersect pairwise,
then all $k$ intersect within some $k$-cube.   
\end{lemma}

\begin{proof}
See \cite[Theorem 4.14]{SageevEndsOfGroups}.
\end{proof}

\begin{lemma}  
\label{lem-k-hplanes}
Assume that  $k$ distinct hyperplanes in a $\cat(0)$ cube complex have
a non-empty intersection. If  are they are all adjacent to a vertex,
then they intersect in a $k$-cube that contains that vertex.   
\end{lemma}

\begin{proof} See  {\cite[Lemma 2.14 and Proposition 2.15]{NibloReevesNormal}}.
\end{proof}

\begin{lemma}
\label{lem-4-quadrants}
If two hyperplanes $H$ and $K$ in a $\cat(0)$ cube complex $X$ are
disjoint, then one of the half-spaces of $H$ is contained in one of the half-spaces of $K$.
\end{lemma}

\begin{proof}
See \cite[Lemma 2.10]{GuentnerHigson}.
\end{proof}


\section{The  Julg-Valette Complex}
\label{sec-jv-cplx}

Let $X$ be a bounded geometry $\cat(0)$ cube complex of dimension $n$.  The aim of this
section is to define a differential complex
\[
 \C[X^0] \stackrel{\dps}{\longrightarrow} 
   \C[X^1]\stackrel{\dps}{\longrightarrow} \cdots  
      \stackrel{\dps}{\longrightarrow} \C[X^{n-1}]
      \stackrel{\dps}{\longrightarrow} \C[X^n] 
\] 
which generalizes the complex introduced by Julg and Valette in the
case of a tree \cite{JulgValetteCRAS,JulgValetteQ_p}.  To motivate the
subsequent discussion we recall their construction.  Let $T$ be a tree
with vertex set $T^0$ and edge set $T^1$.  Fix a base vertex $P_0$.
The Julg-Valette differential
\begin{equation*}
  d: \C[T^0] \longrightarrow \C[T^1]
\end{equation*}
is defined by mapping a vertex $P\neq P_0$ to the first edge $E$ on the
unique geodesic path from $P$ to $P_0$; $P_0$
itself is mapped to zero.  There is  an adjoint differential 
\begin{equation*}
  \delta:  \C[T^1] \longrightarrow \C[T^0]
\end{equation*}
that maps each edge to its furtherest vertex from $P_0$.  The
composite $d\delta$ is the identity on $\C[T^1]$, whereas $1-\delta d$
is the natural rank-one projection onto the subspace of $\C[T^0]$
spanned by the base vertex.  It follows easily that the cohomology of
the Julg-Valette complex is $\C$ in degree zero and $0$ otherwise.

For the higher-dimensional construction we shall need a concept of
orientation for the cubes in $X$, and we begin there. 

\begin{definition}
  A \emph{presentation} of a cube consists of a vertex in the cube,
  together with a linear ordering of the hyperplanes that cut the
  cube.  Two presentations are \emph{equivalent} if the edge-path
  distance between the two vertices has the same parity as the
  permutation between the two orderings.  An \emph{orientation} of a
  cube of positive dimension is a choice of equivalence class of
  presentations; an \emph{orientation} of a vertex is a choice of sign
  $+$ or $-$.
\end{definition}

\begin{remark}
  Every cube has precisely two orientations, and if $C$ is an oriented
  cube we shall write $C^*$ for the same underlying unoriented cube
  equiped with the opposite orientation.
\end{remark}

\begin{definition}
\label{def-cochain}
The space $\C[X^q]$ of \emph{oriented $q$-cochains on $X$} is the
vector space comprising the finitely-supported, \emph{anti-symmetric},
complex-valued functions on the set of oriented $q$-cubes $X^q$.
Here, a function $f$ is anti-symmetric if $f(C)+f(C^*)=0$ for every
oriented cube $C$.
\end{definition}

\begin{remark}\label{new-paradigm}
  The space $\C[X^q]$ is a subspace of the vector space of of
  \emph{all} finitely supported functions on $X^q$, which we shall
  call the \emph{full space of $q$-cochains}.  The formula
  \begin{equation*}
    f^*(C) = f(C^*)
  \end{equation*}
defines an involution on the full space of $q$-cochains.
We shall write $C$ for both
  the Dirac function at the oriented $q$-cube $C$ and for the cube
  itself;  in this way $C$ belongs to the full space of $q$-cochains.
  We shall write $\langle C \rangle$ for the oriented $q$-cochain
  \begin{equation*}
    \langle C \rangle = C-C^*\in \C[X^q],
  \end{equation*}
which is the difference of the Dirac functions at $C$ and $C^*$ (the
two possible meanings of the symbol $C^*$ agree).
\end{remark}

Next, we introduce some geometric ideas that will allow us to define
the Julg-Valette differential in higher dimensions.  The first is the
following generalization of the notion of adjacency introduced in
Definition~\ref{def-vertex-adjacent}.

\begin{definition}
\label{def-adjacency}
A $q$-cube $C$ is \emph{adjacent} to a hyperplane $H$ if it is
disjoint from $H$ and if  there exists a $(q{+}1)$-cube containing $C$
as a codimension-one face that is cut by $H$.
\end{definition}

\begin{lemma}\label{unique-adjacent}
A $q$-cube $C$ is adjacent to a hyperplane $H$ if and only if it is
not cut by $H$ and all of its vertices are adjacent to $H$. 
\end{lemma}

\begin{proof}
Clearly, if the cube $C$ is adjacent to $H$ then so are all of its vertices. 
For the converse, assume that all of the vertices of $C$ are adjacent to
$H$.  By Lemma~\ref{lem-3-hyperplanes} it suffices to show that every
hyperplane $K$ that cuts $C$ must also cross $H$.  For this, let $P$
and $Q$ be vertices of $C$ separated only by $K$, and denote by 
$P^{op}$ and $Q^{op}$ the vertices separated from $P$ and $Q$ only by
$H$, respectively.  These four vertices belong to the four distinct half-space
intersections associated with the hyperplanes $H$ and $K$, so that by
Lemma~\ref{lem-4-quadrants} these hyperplanes intersect. 
\end{proof}

\emph{We shall now fix a base vertex $P_0$ in the complex $X$.}
 
\begin{definition}\label{definition-wedge}
Let $H$ be a  hyperplane in $X$.  Define an operator 
\[
H\wedge\underbar{\,\,\,\,}\,\colon 
   \C[X^q] \longrightarrow \C[X^{q+1}]
\]
as follows.
Let $C$ be  an oriented $q$-cube in $X$.   
\begin{alist}
\item We put $H\wedge C =0$ if   $C$ is not adjacent to $H$.
\item In addition, we put  $H\wedge C =0$ if $C$ is adjacent to $H$,
  but $C$ lies in the same $H$-half-space as the base point  $P_0$. 
\item  If $C$ is adjacent to $H$, and is separated by $H$ from the
  base point, then we define $H\wedge C$  to be  the unique cube 
containing $C$ as a codimension-one face that is cut by $H$. 
\end{alist}
As for the orientations in (c), if $C$ has positive dimension and is
oriented by the vertex $P$, and by the listing on
hyperplanes
$
H_1,\dots, H_{q},
$
then we  orient $H\wedge C$  by the  vertex that is separated from $P$
by the hyperplane $H$ alone, and by the listing of hyperplanes 
$
 H, H_1,\dots, H_q.
$
If $C$ is a vertex with orientation $+$ then $H\wedge C$ is oriented
as above; if $C$ has orientation $-$ then $H\wedge C$ receives the
opposite orientation.
\end{definition}

\begin{remark}\label{ops-in-new-paradigm}
  The linear operator $H\wedge \underbar{\,\,\,\,}$ of the previous
  definition is initially defined on the full space of $q$-cochains by
  specifying its values on the oriented $q$-cubes $C$, which form a
  basis of this space.  We omit the elementary check that for
  an oriented $q$-cube $C$ we have 
  \begin{equation}\label{wedge-involution}
    H\wedge C^* = (H\wedge C)^*,
  \end{equation}
  which allows us to restrict $H\wedge \underbar{\,\,\,\,}$ to an
  operator on the spaces of oriented $q$-cochains.  We shall employ
  similar conventions consistently throughout, so that all linear
  operators will be defined initially on the full space of
  cochains and then restricted to the space of oriented
  cochains.  Some formulas will hold only for the restricted
  operators and we shall point these few instances out.
\end{remark}

\begin{definition} 
\label{def-JV-differential}
The \emph{Julg-Valette differential} is the linear map  
\[
d \colon  \C[X^q] \longrightarrow \C[X^{q+1}]
\]
given by the formula
\[
d \, C  = \sum_{H}  H\wedge C  , 
\]
where the sum is taken over all hyperplanes in $X$. Note that only
finitely many terms in this sum are nonzero.  
\end{definition}

 \begin{example}
   In the case of a tree, if $P$ is any vertex distinct from the base
   point $P_0$, then $H\wedge P $ is the first edge on the geodesic
   edge-path from $P$ to $P_0$ and our operator $d$ agrees the one
   defined by Julg and Valette.  Once a base point is chosen every
   edge (in any $\cat(0)$ cube complex) is canonically oriented by
   selecting the vertex nearest to the base point; vertices are
   canonically oriented by the orientation $+$.  Thus, because the
   original construction of Julg and Valette involves only vertices
   and edges and assumes a base point, orientations do not appear
   explicitly.
\end{example}

\begin{lemma}
\label{lemma-antisymmetry1}
If  $H_1$ and $H_2 $ are any two   hyperplanes, and if $C$ is any oriented cube, then 
\begin{alist}
\item $H_1 \wedge H_2 \wedge C$ is nonzero if and only if $H_1$ and
  $H_2$ are distinct,  they are  both adjacent to $C$, and  they both
  separate $C$ from $P_0$. 
\item
$H_1 \wedge H_2 \wedge  C= (H_2 \wedge H_1 \wedge C)^*$.
\end{alist}
\end{lemma}

\begin{remark}
Here, $H_1 \wedge H_2 \wedge C$ means $H_1 \wedge (H_2 \wedge C)$, and
so on. 
\end{remark}

\begin{proof}
  Item (a) follows from Lemmas~ \ref{lem-k-hplanes} and
  \ref{lem-4-quadrants}.  To prove (b), note first that as a result of
  (a) the left hand side is nonzero if and only if the right hand side
  is nonzero.  In this case, both have the same underlying unoriented
  $(q+2)$-cube, namely the unique cube containing $C$ as a
  codimension-two face and cut by $H_1$ and $H_2$.  As for
  orientation, suppose $C$ is presented by the ordering $K_1, \dots,
  K_q$ and the vertex $P$.  The cube $H_1 \wedge H_2 \wedge C$ is then
  presented by the ordering $H_1, H_2, K_1, \dots, K_q$ and the vertex
  $Q$, the vertex immediately opposite both $H_1$ and $H_2$ from $P$;
  the cube $H_2 \wedge H_1 \wedge C$ is presented by the ordering
  $H_2, H_1, K_1, \dots, K_q$ and the same vertex.  The same argument
  applies when $C$ is a vertex with the orientation $+$, and the
  remaining case follows from this and the identity
  (\ref{wedge-involution}).
\end{proof}

\begin{lemma}
The Julg-Valette differential $\dps$, regarded as an operator on the
space of oriented cochains, satisfies $\dps^2 = 0$.  
\end{lemma}

\begin{proof}
Let $C$ be any $q$-cube, so that
\begin{equation*}
  \dps^2\, \langle C \rangle = 
    \sum_{H_1,H_2}  H_1\wedge H_2 \wedge \langle C \rangle, 
\end{equation*}
As a consequence of Lemma~\ref{lemma-antisymmetry1} we have 
$H_1\wedge H_2\wedge \langle\, C \,\rangle +
  H_2\wedge H_1\wedge \langle\, C \,\rangle = 0$, and the sum
vanishes.  It is important here that we work on
$\C[X^q]$ and not on the larger full space of $q$-cochains, where the
result is not true.  See Remarks~\ref{new-paradigm} and
\ref{ops-in-new-paradigm}.
\end{proof}

\begin{definition}\label{definition-hook}
Let $H$ be a   hyperplane and let $q\ge 1$.  Define an operator 
\[
H\righthalfcup\underbar{\,\,\,\,}\,\colon 
   \C[X^q] \longrightarrow \C[X^{q-1}]
\]
as follows.  Let $C$ be  an oriented $q$-cube in $X$.   
\begin{alist}
  \item If $H$ does not cut $C$, then  $H\righthalfcup C = 0$.  
  \item If $H$ does cut $C$ then we define $H\righthalfcup C$ to be
    the codimension-one face of $C$ that lies entirely in the
    half-space of $H$ that is separated from the base point by $H$.
\end{alist}
As for orientations in (b), if $C$ is presented by the ordered list
$H,H_1,\dots,H_{q-1}$ and the vertex $P$, and $P$ is \emph{not}
separated from the base point  by $H$, then $H\righthalfcup C$ is
presented by the ordered list $H_1,\dots,H_{q-1}$ and the vertex
separated from $P$ by $H$ alone.  If $C$ is an edge
presented by the vertex $P$ \emph{not} separated from the base point  by
$H$ then $H\righthalfcup C=P^{op}$, the vertex of $C$ opposite to $P$,
with the orientation $+$; if $C$ is
presented by the vertex $P$ and $P$ \emph{is} separated from the
base point  by $H$ then $H\righthalfcup C=P$ with the orientation $-$.
\end{definition}

\begin{remark}
For convenience we shall  define the operator
$H\righthalfcup\underbar{\,\,\,\,}$ to be zero on vertices. 
\end{remark}

\begin{example}
  Let us again consider a tree $T$ with a selected base vertex $P_0$.
  If $E$ is any edge then $H\righthalfcup E$ is zero unless $H$ cuts
  $E$.  In this case $H\righthalfcup E=P$, where $P$ is the vertex
  of $E$ which is farthest away from $P_0$; we choose the orientation
  $-$ if $E$ was oriented by the vertex $P$, and the orientation $+$
  otherwise. 
\end{example}

\begin{definition}
\label{def-deltaps}
Let $q\ge 0$.  Define an operator 
\[
\deltaps \colon \C[X^{q+1}] \longrightarrow \C[X^{q}]
\]
 by 
\[
\deltaps \, C = \sum_{H}    H\righthalfcup C .
\]
\end{definition}

\begin{definition}\label{inner-product}
  The oriented $q$-cubes are a vector space basis for the full space
  of $q$-cochains.  We equip this space with an inner product by
  declaring this to be an orthogonal basis and each oriented $q$-cube
  to have length $1/\sqrt{2}$.  The subspace $\C[X^q]$ of oriented
  $q$-cochains inherits an inner product in which
\begin{equation*}
\langle \langle C_1 \rangle, \langle C_2 \rangle \rangle = 
\begin{cases} 
1, & \text{if $C_1 = C_2$}\\
- 1, & \text{if $C_1 = C_2^*$}\\
0, & \text{otherwise}.
\end{cases}
\end{equation*}
Thus, selecting for each unoriented $q$-cube one of its possible
orientations gives a collection of oriented $q$-cubes for which the
corresponding $\langle C \rangle$ form an orthonormal basis of the
space of oriented $q$-cochains; this basis is canonical up to signs
coming from the relations $-\langle C \rangle = \langle C^* \rangle$.
\end{definition}

 \begin{proposition}
   The operators $\dps$ and $\deltaps$ of
   Definitions~\ref{def-JV-differential} and \ref{def-deltaps} are
   formally adjoint and bounded with respect to the inner products in
   Definition~\ref{inner-product}.
\end{proposition}

\begin{proof}
  The fact that the operators are bounded follows from our assumption
  that the complex $X$ has bounded geometry.  The fact that they are
  adjoint follows from the following assertion: for a hyperplane $H$,
  an oriented $q$-cube $C$ and an oriented $(q+1)$-cube $D$ we have
  that $H\wedge C = D$ if and only if $H\righthalfcup D = C$.  
  See Definitions \ref{definition-wedge} and
  \ref{definition-hook}.
\end{proof}

To conclude the section, let us compute the cohomology of the
Julg-Valette complex.   We form the \emph{Julg-Valette Laplacian}
\begin{equation}
\label{eq-jv-laplacian}
\Delta = (d + \delta)^2 = d \delta +  \delta d,
\end{equation}
where all operators are defined on the space of oriented cochains (and
not on the larger full space of cochains), where we have available the
formula $d^2=0$ and hence also $\delta^2=0$.

\begin{proposition}
\label{prop-delta-w-diagonal}
If $C$ is an oriented $q$-cube then
\begin{equation*}
  \Deltaps \langle C \rangle =  \left( q + p(C) \right) \langle C \rangle,
\end{equation*}
where $p(C)$ is the number of hyperplanes that are adjacent to $C$ and
separate $C$ from $P_0$.  In particular, the $\langle C \rangle$ form
an orthonormal basis of eigenvectors of $\Deltaps$, which is
invertible on the orthogonal complement of $\langle P_0 \rangle$
\textup{(}and so also on the space of oriented $q$-cochains for
$q>0$\textup{)}. 
\end{proposition}

\begin{proof}
We shall show that each oriented $q$-cube $C$ is an eigenvector of
$d\delta+\delta d$ acting on the full space of $q$-cochains, with
eigenvalue as in the statement.   If $P$ is a vertex, then 
$d\delta P = 0$ for dimension reasons while  
$\deltaps\dps P = p(P) P$, irrespective of the choice of orientation. 
In higher dimensions, if $q\ge 1$ and $C$ is an oriented $q$-cube,
then 
\begin{equation*}
\deltaps \dps  C = \sum_{H_1, H_2} \, H_1\righthalfcup H_2 \wedge C
\end{equation*}
and similarly 
\begin{equation*}
\dps \deltaps   C =  \sum_{H_1, H_2} \, H_1\wedge H_2 \righthalfcup C.
\end{equation*}
Adding these, and separating the sum into terms where $H_1=H_2$ and
terms where $H_1\ne H_2$ we obtain
\begin{equation}
\label{eqn-kd+dk}
  (d\delta + \delta d) C = \sum_{H}  \left( H \righthalfcup H \wedge C + 
                H \wedge H \righthalfcup  C \right) + 
     \sum_{H_1\ne H_2} \,  \left( H_1\righthalfcup H_2 \wedge C + 
             H_2\wedge H_1 \righthalfcup C \right).
\end{equation}

It follows from Lemma~\ref{lemma-wedge-int} below that (each term of)
the second sum in (\ref{eqn-kd+dk}) is zero.  To understand the first
sum in (\ref{eqn-kd+dk}), observe that if $H$ is any hyperplane and
$C$ is any oriented cube, then
\[
 H \wedge (H\righthalfcup C) = \begin{cases} C,  & \text{if $H$ cuts $C$} \\ 
    0, & \text{otherwise},
    \end{cases}
\]
and also 
\[
H \righthalfcup (H\wedge C)=  
\begin{cases}
C, & \text{if $C$ is adjacent to $H$ and is separated by $H$ from $P_0$}\\
0, & \text{otherwise}.
\end{cases}
\] 
The proposition now follows.
\end{proof}

\begin{lemma}
\label{lemma-wedge-int}
If $H_1$ and $H_2$ are distinct hyperplanes, then
 \[
 H_1\righthalfcup H_2 \wedge C = H_2\wedge H_1\righthalfcup C^*
 \]
 for every oriented cube $C$.
\end{lemma}

\begin{proof}
If $C$ is a vertex then both sides of the formula are zero.  More
generally, if $C$ is a $q$-cube and one of the following two
conditions fails then both sides of the formula are zero:
\begin{alist}
  \item $H_2$ is adjacent to $C$, and separates it from the base point;
  \item $H_1$ cuts $C$ and crosses $H_2$.
\end{alist}
Assume both of these conditions, and suppose that $C$ may be presented
by the listing of hyperplanes $H_1,K_2,\dots,K_q$ and vertex $P$,
\emph{and that $H_1$ separates $P$ from the base point }; if $C$ is not
an edge this is always possible.  We shall leave the exceptional case in
which $C$ is an edge oriented by its vertex closest to the base point 
to the reader.

Now, let $Q$ be the vertex of $C$ separated from $P$ by $H_1$ alone,
and let $P^{op}$ and $Q^{op}$ be the vertices directly opposite $H_2$
from $P$ and $Q$, respectively.  The cube $H\wedge C$ is presented by
the listing $H_2,H_1,K_2,\dots,K_q$ together with the vertex $P^{op}$,
hence also by the listing $H_1,H_2,K_2,\dots,K_q$ and the vertex
$Q^{op}$.  It follows that $H_1\righthalfcup H_2\wedge C$ is presented
by the listing $H_2,K_2,\dots,K_q$ and the vertex $P^{op}$.  As for
the right hand side, $C^*$ is presented by the same listing as $C$ but
with the vertex $Q$, so that $H_1\righthalfcup C^*$ is presented by
the listing $K_2,\dots,K_q$ and the vertex $P$.  It follows that
$H_2\wedge H_1\righthalfcup C^*$ is presented by the listing
$H_2,K_2,\dots,K_q$ and the vertex $P^{op}$, as required.
\end{proof}

\begin{corollary}
The cohomology of the Julg-Valette complex is $\C$ in degree zero and
$0$ otherwise.  
\end{corollary}

\begin{proof}
In degree $q=0$ the kernel of $\dps$ is one
dimensional and is spanned by $\langle P_0 \rangle$.  In
degrees $q\geq 1$ proceed as follows.  From $\dps^2=0$ it follows that
$\dps \Deltaps = \dps\deltaps\dps =\Deltaps \dps$, so that also 
$\dps\Deltaps ^{-1}  = \Deltaps ^{-1}\dps$.   Now the calculation 
\begin{equation*}
  f = \Deltaps \Deltaps^{-1} f =
  (\dps\deltaps+\deltaps\dps)\Deltaps^{-1} f =
  \dps(\deltaps\Deltaps^{-1}) f
\end{equation*}
shows that an oriented $q$-cocycle $f$ is also an oriented $q$-coboundary.
\end{proof}

We conclude the section with a slight generalization that will be needed later.

\begin{definition} 
\label{def-weight-fn-and-differential}
  A \emph{weight function} for $X$ is a positive-real-valued function $w$ on
  the set of hyperplanes in $X$.  The \emph{weighted Julg-Valette
    differential} is the linear map
\[
d_w \colon  \C[X^q] \longrightarrow \C[X^{q+1}]
\]
given by the formula
\[
d_w \, C  = \sum_{H}  \weight(H)\,  H\wedge C . 
\]
In addition the adjoint operator 
\[
\delta_w \colon \C[X^{q+1}] \longrightarrow  \C[X^{q}]
\]
is defined    by 
\[
\delta_w \, C = \sum_{H} \weight(H)\;   H\righthalfcup C .
\]
\end{definition}

\begin{remark}
We are mainly interested
in the following examples, or small variations on them: 
\begin{alist}
\item $\weight (H) \equiv 1$.
\item $\weight (H) = $ the minimal edge-path distance to the base
  point $P_0$ from a vertex adjacent to $H$. 
\end{alist} 
\end{remark}

The calculations in this section are easily repeated in the weighted
context: the operators $d_w$ and $\delta_w$ are formally adjoint,
although \emph{unbounded} in the case of an unbounded weight function
as, for example, in (b); both are differentials when restricted to the
spaces of oriented cochains; and the
cohomology of either complex is $\C$ in degree zero and $0$ otherwise.
We record here the formula for the weighted Julg-Valette Laplacian.
Compare Proposition~\ref{prop-delta-w-diagonal}.

\begin{proposition}
\label{prop-delta-w-diagonal2}
If $C$ is an oriented $q$-cube then 
\begin{equation*}
  \Delta_w \langle C \rangle =  ( q_w(C) + p_w(C) ) \langle C \rangle,
\end{equation*}
where $q_w(C)$ is the sum of the squares of the weights of the hyperplanes that
cut $C$ and $p_w(C)$ is the sum of the squares of the weights
of the hyperplanes that are adjacent to $C$ and separate $C$ from the
base vertex.  \qed
\end{proposition}


\section{Parallelism Classes of Cubes}
\label{sec-parallelism}
 
The remaining aspects of our generalization of the Julg-Valette and
Pytlik-Szwarc theory to $\cat(0)$ cube complexes all rest on the
following geometric concept:  
 
\begin{definition}
\label{def-parallel}
Two   cubes $D_1$ and $D_2$ in a $\cat(0)$ cube complex $X$ are
\emph{parallel} if they have the same dimension, and if every
hyperplane that cuts $D_1$   also cuts $D_2$.  
\end{definition} 

 \begin{figure}[ht] 
    \centering
    \includegraphics[scale=0.8]{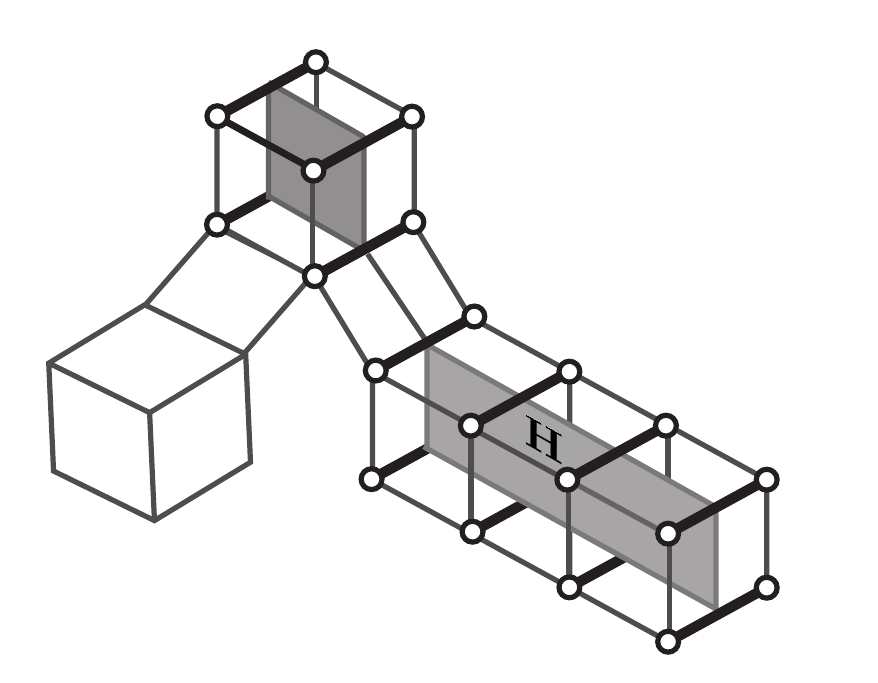} 
       \caption{ The darker edges form a parallelism class determined by the hyperplane $H$, see Definition~\ref{def-parallel}.} 
   \label{fig:parallel}
 \end{figure}

Every parallelism class  of $q$-cubes in $X$  is determined by, and
determines, a set  of $q$ pairwise intersecting hyperplanes, namely
the hyperplanes that cut all the cubes in the parallelism class.
Call these the \emph{determining hyperplanes} for the parallelism
class.  
 
\begin{proposition}\label{general-Sageev}
The intersection of the determining hyperplanes associated to a
parallelism class of $q$-cubes carries the structure of  a $\cat(0)$ cube
complex in which the $p$-cubes are the intersections of  this space
with the $(p{+}q)$-cubes in $X$ that are cut by every determining
hyperplane. 
\end{proposition}

\begin{proof}
  The case when $q=0$ is the assertion that $X$ itself is a $\cat(0)$
  cube complex.  The case when $q=1$ is the assertion that a
  hyperplane in $\cat(0)$ cube complex $X$ is itself a $\cat(0)$ cube
  complex in the manner described above, and this is proved by Sageev
  in \cite[Thm.~4.11]{SageevEndsOfGroups}.
  
  For the general result, we proceed inductively as follows.  Suppose
  given $k$ distinct hyperplanes $K_1,\dots,K_k$ in $X$.  The
  intersection $Z=K_2\cap\dots\cap K_k$ is then a $\cat(0)$ cube
  complex as described in the statement, and the result will follow
  from another application of \cite[Thm.~4.11]{SageevEndsOfGroups}
  once we verify that $K_1\cap Z$ is a hyperplane in $Z$.  Now the
  cubes, and so also the midplanes of $Z$ are exactly the non-empty
  intersections of the cubes and midplanes of $X$ with $Z$.  So, we
  must show that if two midplanes belonging to the hyperplane $K_1$ of
  $X$ intersect $Z$ non-trivially then their intersections are
  hyperplane equivalent \emph{in $Z$}.  But this follows from the fact
  that $Z$ is a totally geodesic subspace of $X$.
 
\end{proof}

\begin{proposition}
\label{prop-nearest-cube}
Let $X$ be a $\cat(0)$ cube complex and let $P$  be a   vertex in $X$.
In each parallelism class of $q$-cubes there is   a unique cube that
is closest to   $P$, as measured by   the distance from  closest point
in the cube  to $P$  in the edge-path metric. 
\end{proposition}

Before beginning the proof, we recall that the edge-path distance
between two vertices is equal to the number of hyperplanes separating
the vertices; see for example \cite[Theorem 4.13]{SageevEndsOfGroups}.
In addition, let us make note of the following simple fact: 

\begin{lemma}
\label{lem-cuts-determining}
A hyperplane that separates two vertices of distinct cubes in the same
parallelism class must intersect every determining hyperplane. 
\end{lemma}

\begin{proof}
This is obvious if the hyperplane is one of the determining
hyperplanes.  Otherwise, the hyperplane must in fact separate two
cubes in the parallelism class, and so it must separate two midplanes
from each determining hyperplane.  Since hyperplanes are connected the
result follows. 
\end{proof}

\begin{proof}[Proof of Proposition~\ref{prop-nearest-cube}]
Choose a  vertex $R$ from among the cubes in the parallelism class such that 
\begin{equation}
\label{eq-distinguished-minimizer}
d(P,R)  \le d (P,S)
\end{equation}
for every other such vertex $S$.  We shall prove the \emph{addition formula} 
\begin{equation}
\label{eq-addition-fmla}
d(P,S) = d(P,R) + d(R,S),
\end{equation}
and this will certainly prove the uniqueness of $R$.

The  addition formula \eqref{eq-addition-fmla} is  a consequence of the following \emph{hyperplane property} of any   $R$ satisfying \eqref{eq-distinguished-minimizer}: 
 \emph{every  hyperplane that separates $P$ from $R$  is parallel to \textup{(}that is, it does not intersect\textup{)} at least one determining hyperplane.}
Indeed,  it follows from Lemma~\ref{lem-cuts-determining} and  
the hyperplane property that no hyperplane
can separate $R$ from both $P$ and $S$, so that  (\ref{eq-addition-fmla}) follows from the
characterization of the edge path distance given above.

It remains to prove the hyperplane property for any $R$ satisfying \eqref{eq-distinguished-minimizer}. For this we shall use the notion of \emph{normal cube path} from \cite[Section 3]{NibloReevesNormal}.
There exists a normal cube path from $R$ to $P$ with vertices
\[
R=R_1,\dots ,R_l=P .
\]
This means that every pair of consecutive $R_i$ are diagonally opposite a cube, called a \emph{normal cube}, all of whose hyperplanes separate $R$ from $P$, and every such separating hyperplane cuts exactly one normal cube.  It also means that  every hyperplane $K$ separating $R_i$ from $R_{i+1}$ is parallel to at least one of the hyperplanes $H$ separating $R_{i-1}$ from $R_{i}$ (so each normal cube is, in turn, as large as possible).  Note that the hyperplane $K$  is  contained completely in the half-space of $H$ that contains $P$.

No hyperplane $H$ separating $R=R_1$ from  $R_2$ can intersect every
determining hyperplane, for if it did, then it would follow from
Lemma~\ref{lem-k-hplanes} that $H$  and the determining hyperplanes
would intersect in a $(q{+}1)$-cube having $R$ as a vertex. The vertex
$S$ separated from $R$ by $H$ alone would then belong to a cube in the
parallelism class, and would be
strictly closer to $P$ than $R$.      

Consider the second normal cube, with opposite vertices $R_3$ and $R_2$.  Any hyperplane $K$ separating $R_3$ from $R_2$ is parallel to some hyperplane $H$ separating $R_2$ from $R_1$, and this is in turn parallel to some determining hyperplane.  But $K$   is  contained completely in the half-space of $H$ that contains $P$, while the determining hyperplane is contained completely in the half-space of $H$ that contains $R$.  So $K$ does not meet this determining hyperplane.

Continuing in this fashion with successive normal cubes, we find that every hyperplane that separates $P$ from $R$ is indeed parallel to some determining hyperplane, as required.
\end{proof}

We can now verify the formula mentioned in the introduction:

\begin{proposition}
  If X is finite $\cat(0)$ cubical space, then the number of vertices
  of $X$ is equal to the number of parallelism classes of cubes of all
  dimensions. 
\end{proposition}

\begin{proof}
Fix a base vertex $P$ and associate to each vertex $Q$ the first cube
in the normal cube path from $Q$ to $P$.  This correspondence induces
a bijection from vertices to parallelism classes of cubes.   

Indeed it follows from the hyperplane property that if $C$ is the nearest cube to $P$ within its parallelism class, and if $Q$ is the vertex of $C$ furthest from $P$, then $C$ is the first cube in the normal cube path from $Q$ to $P$. So our map is surjective. On the other hand it follows from the addition formula that if $C$ is \emph{not} nearest to $P$ within its equivalence class, and if $Q$ is the vertex of $C$ furthest from $P$, then any hyperplane that  separates $Q$ from the nearest cube also separates $Q$ from $P$.  Choosing a hyperplane that is adjacent to $Q$ but not does not cut $C$, we find that  $C$ is not the first cube in the normal cube path from $Q$ to $P$, and our map is injective.
\end{proof}

\begin{proposition}
\label{prop-nearest-cube2}
Let $X$ be a $\cat(0)$ cube complex and let $P$ and $Q$ be     vertices  in $X$ that  are separated by a single hyperplane $H$. The nearest $q$-cubes to $P$ and $Q$ within a parallelism class are either the same, or are opposite faces, separated by  $H$, of a $(q{+}1)$-cube that is cut by $H$.
\end{proposition}

\begin{proof}
Denote by $R$  and $S$ the nearest  vertices to $P$ and $Q$, respectively,  among the vertices of cubes in the equivalence class, and suppose  that a hyperplane $K$ separates $R$ from $S$. Then it must separate $P$ from $S$ by the addition formula \eqref{eq-addition-fmla} applied to the nearest point $R$, and also separate $Q$ from $R$, by the addition formula applied to the nearest point $S$.  So it  must separate $P$ from $Q$, and hence must be $H$.  So either there is no hyperplane separating $R$ from $S$, in which case of course $R=S$ and the nearest cubes to $P$ and $Q$ are the same, or $R$ is opposite $S$ across $H$.  If $H$ is a determining hyperplane, then $R$ and $S$ are vertices of the same $q$-cube in the parallelism class; if $H $ is not a determining hyperplane, then $R$ and $S$ belong to $q$-cubes that are opposite to one another across $H$, as required. 
\end{proof}



\section{The Pytlik-Szwarc Complex}
\label{sec-eq-complex}

As described in the introduction, our ultimate goal involves deforming
the Julg-Valette complex into what we call the \emph{Pytlik-Szwarc
  complex\/}, a complex with the same cohomology but which is
equivariant in the case of a group acting on the $\cat(0)$ cube
complex.  In this short section we describe the (algebraic)
Pytlik-Szwarc complex.

As motivation for what follows we consider how to compare orientations
on parallel cubes.  The key observation is that a vertex in a $q$-cube
is uniquely determined by its position relative to the cutting
hyperplanes $K_1,\dots,K_q$.  Thus, there is a natural isometry
between (the vertex sets of) any two parallel $q$-cubes.  We shall say
that parallel $q$-cubes of positive dimension are \emph{compatibly
  oriented} if their orientations are presented by vertices $P_1$ and
$P_2$ which correspond under this isometry and a common listing of the
cutting hyperplanes $K_1,\dots,K_q$; vertices are \emph{compatibly
  oriented} if they are oriented by the same choice of sign.

We shall now generalize these considerations to pairs comprising a
cube and one of its faces.

\begin{definition}
\label{def-parallel-pairs}
A \emph{cube pair} is a pair $(C,D)$ in which $C$ is a cube containing
$D$ as a face.  Two cube pairs $(C_1,D_1)$ and $(C_2,D_2)$ are
\emph{parallel} if the cubes $C_1$ and $C_2$ are parallel, and the
cubes $D_1$ and $D_2$ are parallel too.  When $D$ is a $q$-cube, and
$C$ is a $(p+q)$-cube, we shall call $(C,D)$ a $(p,q)$-cube pair,
always keeping in mind that in this notation $p$ is the
\emph{codimension} of $D$ in $C$.
\end{definition}

We may describe the parallelism class of a
$(p,q)$-cube pair $(C,D)$ by grouping the determining hyperplanes of
the parallelism class of $C$ into a symbol
\begin{equation}
\label{unoriented-symbol}
      \left\{\, H_1,\dots,H_p \,|\, K_1,\dots,K_q \,\right\},
\end{equation}
in which the $K_1, \dots, K_q$ determine the parallelism class of $D$.
The hyperplanes $H_1, \dots, H_p$ which cut $C$ but not $D$ are the
\emph{complementary hyperplanes} of the cube pair, or of the
parallelism class.

An \emph{orientation} of a cube pair $(C,D)$ is an orientation of the
face $D$.  In order to compare orientations of parallel cube pairs
$(C_i,D_i)$ we can compare the orientations on the faces $D_i$, which
are themselves parallel cubes, but must also take into account the
position of the faces within the ambient cubes $C_i$.  For this we
introduce the following notion.

\begin{definition}\label{def-pair-parity}
Two parallel  cube pairs $(C_1,D_1)$ and $(C_2,D_2)$  have the
\emph{same parity} if the number of complementary hyperplanes
that separate $D_1$ from $D_2$, is
\emph{even}.   Otherwise they have the \emph{opposite parity}. 
\end{definition}

\begin{definition}\label{def-equivalent-presentations}
  Let $(C_1,D_1)$ and $(C_2,D_2)$ be parallel cube pairs, each
  with an orientation.  The orientations are \emph{aligned}
  if one of the following conditions holds:
\begin{alist}
\item   $(C_1,D_1)$ and $(C_2,D_2)$ have the  {same parity}, and 
  $D_1$ and $D_2$ are compatibly oriented; or
\item  $(C_1,D_1)$ and $(C_2,D_2)$ have the  {opposite parity}, and 
  $D_1$ and $D_2$ are not compatibly oriented.
\end{alist}
\end{definition}

In the symbol (\ref{unoriented-symbol}) describing the parallelism
class of a cube pair $(C,D)$, the hyperplanes are not ordered; the
only relevant data is which are to the left, and which to the right of
the vertical bar.  
If the cube pair $(C,D)$ is oriented, then the symbol receives
additional structure coming from the orientation of $D$.  We group the
determining hyperplanes as before, and include a vertex $R$ of $D$ into
a new symbol
\begin{equation}
\label{oriented-symbol}
   \left\{\, H_1,\dots,H_p \,|\, K_1,\dots,K_q \,|\, R \,\right\}.
\end{equation}
Here, in the case $q>0$, the hyperplanes $K_1,\dots,K_q$ form an
\emph{ordered list} which, together with the vertex $R$ are a
presentation of the oriented cube $D$.  In the case $q=0$ this\ list is
empty and we replace it by the sign representing the orientation of
the vertex $D=R$, obtaining a symbol of the form
\begin{equation}
  \label{oriented-symbol-vrtx}
  \left\{\, H_1,\dots,H_p \,|\, + \,|\, R \,\right\} 
    \quad\text{or}\quad
  \left\{\, H_1,\dots,H_p \,|\, - \,|\, R \,\right\}. 
\end{equation}
In either case the hyperplanes $H_1,\dots,H_p$ remain an
\emph{unordered set}.  Conversely, a formal expression as in
(\ref{oriented-symbol}) or (\ref{oriented-symbol-vrtx}) is the symbol
of some oriented $(p,q)$-cube pair precisely when the hyperplanes
$H_1,\dots,K_q$ are distinct and have nonempty (pairwise)
intersection, and the vertex $R$ is adjacent to all of them.

The following definition captures the notion of alignment of
orientations in terms of the associated symbols.

\begin{definition}\label{def-symbol-orientation}
Symbols 
\begin{equation*}
  \left\{\, H_1,\dots,H_p \,|\, K_1,\dots,K_q \,|\, R \,\right\} 
    \quad\text{and}\quad
  \left\{\, H'_1,\dots,H'_p \,|\, K'_1,\dots,K'_q \,|\, R' \,\right\}
\end{equation*}
of the form (\ref{oriented-symbol}) are \emph{equivalent} if 
\begin{alist}
  \item the sets $\{\, H_1,\dots,H_p \,\}$ and $\{\, H'_1,\dots,H'_p \,\}$ are equal; 
  \item the $K_1,\dots,K_q$ are a permutation of the
    $K'_1,\dots,K'_q$; and 
  \item the number of hyperplanes among the $H_1,\dots,K_q$ separating
    $R$ and $R'$ has the same parity as the permutation in (b).
\end{alist}
In the case of symbols of the form (\ref{oriented-symbol-vrtx}) we
omit (b) and replace (c) by
\begin{alist}
\item[(c$'$)] the number of hyperplanes among the $H_1,\dots,H_p$ separating
  $R$ and $R'$ is even if the orientation signs agree, and odd otherwise.
\end{alist}
An \emph{oriented $(p,q)$-symbol} is an equivalence class of symbols.
We shall denote the equivalence class of the symbol
(\ref{oriented-symbol}) by
\begin{equation*}
    [\, H_1,\dots,H_p \,|\, K_1,\dots,K_q \,|\, R \,],
\end{equation*}
or simply by $[\, H \,|\, K \,|\, R \,]$ when no confusion can arise,
and we use similar notation in the case of symbols of the form
(\ref{oriented-symbol-vrtx}). 
We shall denote the set of oriented $(p,q)$-symbols by $\H^p_q$, and
the (disjoint) union $\H_q^0 \cup \cdots \cup \H_q^{n-q}$ by $\H_q$.
\end{definition}

\begin{proposition}
  The oriented symbols associated to oriented $(p,q)$-cube pairs agree
  precisely when the orientations of the cube pairs are aligned.  \qed
\end{proposition}

Our generalization of the Pytlik-Szwarc complex will be a differential
complex designed to capture the combinatorics of oriented, aligned
cube pairs:
\begin{equation}
\label{eq-ps-cplx}
   \C[\H_0]\stackrel{d} \longrightarrow \C[\H_1]
     \stackrel{d}\longrightarrow 
         \cdots  \stackrel{d}\longrightarrow \C[\H_{n-1}]
         \stackrel{d}\longrightarrow \C[\H_n].
\end{equation}

\begin{definition}
\label{def-ps-cochain}
The space of oriented $q$-cochains of type $p$ in the Pytlik-Szwarc
complex is the space of finitely supported, \emph{anti-symmetric},
complex-valued functions on $\H^p_q$.  Here, a function is
anti-symmetric  if
\begin{equation*}
  f([\, H \,|\, K \,|\, R \,]) + 
        f([\, H \,|\, K \,|\, R \,]^*) = 0,
\end{equation*}
where we have used the involution on $\H^p_q$ defined by reversing the
orientation of the symbol.  
We shall denote this space by $\C[\H^p_q]$.  The space of oriented
$q$-cochains is defined similarly using the oriented symbols of type
$(p,q)$ for all $0\leq p \leq n-q$.  It splits as the direct sum
  \begin{equation*}
      \C[\H_q]=\C[\H_q^0]\oplus\dots\oplus\C[\H_q^{n-q}].
  \end{equation*}
\end{definition}

\begin{remark}
\label{ps-cochain-remark}
  As with the Julg-Valette cochains, the space of oriented
  Pytlik-Szwarc $q$-cochains of type $p$ is a subspace of the
  \emph{full space of Pytlik-Szwarc $q$-cochains of type $p$}, which
  is the vector space of \emph{all} finitely supported functions on
  the set $\H_q^p$.  We shall follow conventions similar to those in
  Section~\ref{sec-jv-cplx}: we write
  \begin{equation*}
      \left[\, H_1,\dots,H_p \,|\, K_1,\dots,K_q \,|\, R  \,\right]
      \quad\text{or}\quad
      \left[\, H \,|\, K \,|\, R  \,\right]
  \end{equation*}
  for both the Dirac function at an oriented symbol and the symbol
  itself, and 
\begin{equation*}
     \left\langle\, H \,|\, K \,|\, R \,\right\rangle
       = \left[\, H \,|\, K \,|\, R \,\right] -
          \left[\, H \,|\, K \,|\, R \,\right]^* 
                  \in \C[\H_q^p]
\end{equation*}
for the difference of the Dirac functions.  Further, linear operators
will be defined on the full space of cochains by
specifying their values on the basis of Dirac functions at the
oriented symbols.  We shall typically omit the elementary check that
an operator commutes with the involution and so restricts to an
operator on the spaces of oriented cochains.
\end{remark}

We now define the differential in the Pytlik-Szwarc complex
(\ref{eq-ps-cplx}).

\begin{definition}
\label{def-ps-differential}
The Pytlik-Szwarc differential is the linear map
$d:\C[\H_q]\to\C[\H_{q+1}]$ which is $0$ on oriented symbols of type
$(0,q)$ and which satisfies
\begin{equation*}
  d  \left[\, H_1,\dots,H_p \,|\, K_1,\dots,K_q \,|\, R \,\right] =
            \sum_{i=1}^p \,
         \left[\, H_1,\dots,\widehat{H_i},\dots,H_p \,|\, H_i,K_1,\dots,K_q \,|\, R_i \,\right]
\end{equation*}
for oriented $(p,q)$-symbols with $p, q\geq 1$.  Here, $R_i$ is the
vertex separated from $R$ by $H_i$ alone and, as usual, a `hat'
means that an entry is removed.  When $q=0$ the same formula is used
for symbols of the form $[H\,|+|\,R]$ which, together with the requirement
that $d$ commute with the involution, determines $d$ on symbols of the
form $[H\,|-|\,R]$.  Since $d$ maps an oriented symbol of type $(p,q)$ to
a linear combination of oriented symbols of type $(p-1,q+1)$ in all
cases, it splits as the direct sum of linear maps
\begin{equation*}              
  d:\C[\H^p_q] \longrightarrow                   \C[\H^{p-1}_{q+1}]
\end{equation*}
for $0< p\leq n-q$, and is $0$ on the $\C[\H^0_q]$. \end{definition}

\begin{lemma}
  The Pytlik-Szwarc differential $d$, regarded as an operator on the
  space of oriented cochains, satisfies $d^2=0$.  \qed 
\end{lemma}

\begin{example}
Let $T$ be a tree.  The Pytlik-Szwarc complex has the form
\begin{equation*}
    d: \C\oplus \C[\H^1_0] \longrightarrow \C[\H^0_1],
\end{equation*}
where $d$ is $0$ on $\C$ and, after identifying each of $\C[\H^1_0]$
and $\C[\H^0_1]$ with the space of finitely supported functions on the
set of edges of $T$, the identity $\C[\H^1_0]\to \C[\H^0_1]$.  For
the identifications, note that both $\H^0_1$ and $\H^1_0$ are
identified with the set of oriented edges in $T$ and that the
involution acts by reversing the orientation.  So the space
of anti-symmetric functions on each identifies with the
space of finitely supported functions on the set of edges.
\end{example}

Our goal for the remainder of this section is to analyze the
Pytlik-Szwarc complex.  Emphasizing the similarities with the
Julg-Valette complex we begin by providing a formula for the formal
adjoint of the Pytlik-Szwarc differential.

\begin{definition}
\label{def-ps-delta}
Let $\delta:\C[\H_q] \to \C[\H_{q-1}]$ be the linear map which is $0$
on oriented symbols of type $(p,0)$ and which satisfies 
\begin{equation*}
   \delta  \left[\, H_1,\dots,H_p \,|\, K_1,\dots,K_q \,|\, R \,\right]   =      
   \sum_{j=1}^q (-1)^{j} \left[\, H_1,\dots,H_p,K_j \,|\, 
                K_1,\dots,\widehat{K_j},\dots,K_q \,|\, R \,\right],
\end{equation*}
for oriented symbols of type $(p,q)$ with $q\geq 1$.  
Again a `hat' means that an entry is removed.  Since $\delta$ maps an
oriented symbol of type $(p,q)$ to a linear combination of oriented
symbols of type $(p+1,q-1)$ it splits as a direct sum of linear maps
\begin{equation*}
  \delta:\C[\H^p_q] \to \C[\H^{p+1}_{q-1}]
\end{equation*}
for $0<q\leq n-p$, and is $0$ on the $\C[\H^p_0]$.
\end{definition}

\begin{definition}\label{def-ps-inner-product}
  We define an inner product on the full space of Pytlik-Szwarc
  $q$-cochains by declaring that the elements of $\H_q$ are
  orthogonal, and that each has length $1/\sqrt{2}$.  The subspace
  $\C[\H_q]$ of oriented Pytlik-Szwarc $q$-cochains inherits an inner
  product in which
  \begin{equation*}
   \left\langle
      \left\langle\, H \,|\, K \,|\, R \,\right\rangle ,
      \left\langle\, H' \,|\, K'  \,|\, R' \,\right\rangle
   \right\rangle =
   \begin{cases}
       1, &  [\, H \,|\, K \,|\, R \,] =  [\, H' \,|\, K'  \,|\, R' \,] \\
       -1, &  [\, H \,|\, K  \,|\, R \,] = [\, H' \,|\, K'  \,|\, R' \,]^* \\
       0, & \text{otherwise}
   \end{cases}
  \end{equation*}
\end{definition}

\begin{lemma}
  The operators $\dps$ and $\deltaps$ of
   Definitions~\ref{def-ps-differential} and \ref{def-ps-delta} are
   formally adjoint and bounded with respect to the inner products in
   Definition~\ref{def-ps-inner-product}. \qed
\end{lemma}

\begin{proposition}
\label{prop-ps-laplacian}
The Pytlik-Szwarc Laplacian
\[
 \Delta = (d+\delta)^2=d\delta + \delta d : 
     \C[\H_q] \longrightarrow \C[\H_q]
\]
acts on the summand $\C [\H^p_q]$ as scalar multiplication by $p+q$. 
\end{proposition}

\begin{proof}
  We prove the above statement for the operator $d\delta + \delta d$ defined on the full space of cochains. This operator equals 
  $\Delta$ when restricted to the subspace of oriented cochains. The proof is a direct calculation.  The result of applying 
  $\delta d$ to an oriented symbol 
$[\, H_1,\dots,H_p \,|\, K_1,\dots,K_q \,|\, R \,]$ of type $(p,q)$ 
is the sum 
\begin{align*}
     \sum_{i=1}^p [\, H_1,\dots,H_p \,|\,  &K_1,\dots,K_q \,|\, R \,] \ +\\
     &+\sum_{i=1}^p \sum_{j=1}^q (-1)^j 
        [\, H_1,\dots,\widehat H_i,\dots,H_p,K_j \,|\, H_i,K_1,\dots,\widehat
                  K_j,\dots, K_q \,|\, R\,],
\end{align*}
whereas the result of applying $d\delta$ is
\begin{align*}
       \sum_{j=1}^q (-1)^{j+1} [\, H_1,&\dots,H_p \,|\, 
               K_j,K_1,\dots,\widehat K_j,\dots,K_q \,|\, R \,] \ + \\
     &+\sum_{j=1}^q (-1)^{j+1} \sum_{i=1}^p 
        [\, H_1,\dots,\widehat H_i,\dots,H_p,K_j \,|\, H_i,K_1,\dots,\widehat
        K_j,\dots, K_q \,|\, R \,].
\end{align*}
When these are added, the second summands cancel and the first
summands combine to give 
$(p+q) [\, H_1,\dots,H_p \,|\, K_1,\dots,K_q \,|\, R \,]$.
\end{proof}

\begin{corollary}
\label{cor-ps-cohomology}
  The cohomology of the Pytlik-Szwarc complex is $\C$ in dimension
  zero and $0$ otherwise. \qed
\end{corollary}

  
\section{Continuous Fields of Hilbert Spaces}
\label{sec-cont-field}

Our objective over the next several sections is to construct a family
of complexes that continuously interpolates between the Julg-Valette
complex and the Pytlik-Szwarc complex. We shall construct the
interpolation within the Hilbert space context, using the concept of a
\emph{continuous field} of Hilbert spaces.

We refer the reader to \cite[Chapter 10]{Dixmier} for a comprehensive
treatment of continuous fields of Hilbert spaces.  In brief, a
continuous field of Hilbert spaces over a topological space $T$ consists of  a
family of Hilbert spaces parametrized by the points of $T$, together with a distinguished
family $\Sigma$ of sections that satisfies several axioms, of which the most
important is that the pointwise inner product of any two
sections in $\Sigma$  is a \emph{continuous} function on $T$.  See
\cite[Definition~10.1.2]{Dixmier}.  The following theorem gives a
convenient means of constructing continuous fields. 

\begin{theorem} 
\label{continuous-field}
Let $T$ be a topological space, let $\{{\mathfrak H}_t\} $ be a family
of Hilbert spaces parametrized by the points of $T$, and let $\Sigma_0$ be a
family of sections that satisfies the following conditions:
\begin{alist}
\item The pointwise inner product of any two sections in $\Sigma_0$
  is a continuous function on $T$. 
\item For every $t\in T$ the linear span of  
  $\{ \sigma(t)\, : \, \sigma\in\Sigma \, \}$ is dense in ${\mathfrak H}_t$.  
\end{alist}
There is a unique enlargement of $\Sigma_0$ that gives 
$\{{\mathfrak H}_t\}_{t\in T}$ the structure of a continuous field of
Hilbert spaces. 
\end{theorem}

\begin{proof}
  The enlargement $\Sigma$  consists of all sections $\sigma$
  such that for every $t_0\in T$ and every $\varepsilon >0$ there is a
  section $\sigma_0$ in the linear span of $\Sigma_0$ such that
\[
 \|\sigma_0(t) - \sigma(t)\|_t < \varepsilon
 \]
 for all $t$ in some neighborhood of $t_0$.  See
 \cite[Proposition~10.2.3]{Dixmier}.    
\end{proof}

\begin{definition}\label{generating-sections}
We shall call a family  $\Sigma_0$, as in the statement of
Theorem~\ref{continuous-field}, a \emph{generating family} of sections
for the associated continuous field of Hilbert spaces.  
\end{definition}

Ultimately we shall use the parameter space $T = [0,\infty]$, but in
this section we shall concentrate on the open subspace $(0,\infty]$,
and then extend to $[0,\infty]$ in the next section. In both this
section and the next we shall deal only with the construction of
continuous fields of Hilbert spaces; we shall construct the
differentials acting between these fields in
Section~\ref{sec-field-ps-ops}.

We begin by completing the various cochain spaces from
Section~\ref{sec-jv-cplx} in the natural way so as to obtain Hilbert
spaces. 

\begin{definition}
\label{def-analytic-JV}
Denote by $\ell^2 (X^q)$ the Hilbert space completion of the
Julg-Valette oriented cochain space $\C[X^q]$ in the inner product of
Definition~\ref{inner-product} in which the basis comprised of the
oriented cochains $\langle\, C \,\rangle$ is orthonormal.
\end{definition}

\begin{remark}
  As was the case in Section~\ref{sec-jv-cplx}, we shall also consider
  the full cochain space comprised of the square-summable functions on
  the set of oriented $q$-cubes.  This is the completion of the full
  space of Julg-Valette $q$-cochains in the inner product of
  Definition~\ref{inner-product}, and contains the space $\ell^2(X^q)$
  of the previous definition as the subspace of anti-symmetric
  functions.
\end{remark}

We shall now construct, for every $q\geq 0$, families of Hilbert
spaces parametrized by the topological space $(0,\infty]$.  These will
be completions of the spaces of Julg-Valette $q$-cochains, both full
and oriented, but with respect to a family of pairwise distinct inner
products.  Considering the oriented cochains, we obtain a family of
Hilbert spaces $\ell_t^2(X^q)$ each of which is a completion of the
corresponding $\C[X^q]$.  The Hilbert space $\ell^2_\infty(X^q)$ will
be the space $\ell^2(X^q)$ just defined.

\begin{definition}
\label{def-distance-parallel-cubes}
If $D_1$ and $D_2$ are $q$-cubes in $X$, and if $D_1$ and $D_2$ are parallel and have compatible orientations, then denote by $d(D_1,D_2)$ the number of hyperplanes in $X$ that are disjoint from $D_1$ and $D_2$ and  that separate $D_1$ from $D_2$.  If $D_1$ and $D_2$ are $q$-cubes in $X$, but are not parallel, or 
have incompatible orientations, then set $d(D_1,D_2)=\infty$.
\end{definition}

If $D_1$ and $D_2$ are (compatibly oriented) vertices, then
$d(D_1,D_2)$ is the edge-path distance from $D_1$ to $D_2$.  In higher
dimensions, if $D_1$ and $D_2$ are parallel then they may be
identified with vertices in the $\cat(0)$ cube comples which is the
intersection of the determining hyperplanes for the parallelism class.
If in addition they are compatibly oriented, then $d(D_1,D_2)$ is the
edge-path distance in this complex. Compare
Theorem~\ref{general-Sageev}.

\begin{definition}\label{def-inner-product}
Let $t > 0$ and $q\geq 0$. For every two oriented $q$-cubes $D_1$ and
$D_2$ define  
\[
\bigl\langle D_1,D_2\bigr\rangle_t = 
    \tfrac{1}{2} \exp \bigl(-\tfrac 12 t^2 d(D_1,D_2)\bigr ),
\]
where of course we set $\exp (-\tfrac 12 t^2 d(D_1,D_2) ) = 0$ if
$d(D_1,D_2) = \infty$, and then extend by linearity to a sesqui-linear
form on the full space of Julg-Valette $q$-cochains.
\end{definition}

Note that the formula in the definition makes sense when $t=\infty$,
where
\[
\tfrac{1}{2} \exp \bigl(-\tfrac 12 t^2 d(D_1,D_2)\bigr ) = 
      \begin{cases} \tfrac{1}{2}, & D_1 = D_2 \\ 0, & D_1 \ne D_2.
\end{cases}
\]
In particular, the form $\langle \,\, ,\, \rangle_\infty$ is the one
underlying Definition~\ref{inner-product} that we used to define
$\ell^2 (X^q)$.

\begin{theorem}\label{thm-inner-product}
The sesqui-linear form $\langle\,\,,\,\rangle_t$ is positive semi-definite.
\end{theorem}

\begin{proof}
  Consideration of oriented, as opposed to unoriented, cubes merely
  gives two (orthogonal copies) of each space of functions.  Aside
  from this, the result is proved in 
  \cite[Technical Lemma, p.6]{NibloReevesNormal} in the case $q=0$.
  See also \cite[Prop.~3.6]{GuentnerHigson}.  The case $q>0$ reduces
  to the case $q=0$ using Theorem~\ref{general-Sageev}.
\end{proof}

\begin{definition}
  For $t\in (0,\infty]$ denote by $\ell^2_t(X^q)$ the Hilbert space
  completion of the Julg-Valette oriented cochain space $\C[X^q]$ in
  the inner product $\langle\,\,\, , \,\, \rangle_t$.  

\end{definition}

\begin{remark} 
  The Hilbert spaces of the previous definition are completions of the
  \emph{quotient of $\C[X^q]$ by the elements of zero norm}.  We shall
  soon see that every nonzero linear combination of oriented $q$-cubes
  has nonzero $\ell^2_t$-norm for every $t$, so the natural maps from
  $\C[X^q] $ into the $\ell^2 _t(X^q)$ are injective.
\end{remark} 

Next, we define a generating family of sections, using either one of
the following lemmas; on the basis of Theorem~\ref{continuous-field},
it is easy to check that the continuous fields arising from the lemmas
are one and the same.

\begin{lemma}
\label{lemma-cont-fld1}
Let $t\in (0,\infty]$. The set of all   sections of the form
\[
t\mapsto f\in \C[X^q]\subseteq \ell^2 _t (X^q) ,
\]
indexed by all $f\in \C[X^q]$, is a generating family of sections for a continuous field. \qed
\end{lemma}

 \begin{lemma}
 \label{lemma-cont-fld2}
The set of all   sections of the form 
\[
t\mapsto f(t)\, \langle C \rangle\in  \ell^2 _t (X^q) ,
\]
where $f$ is a continuous scalar function on $(0,\infty]$ and $C$ is
an oriented $q$-cube, is a generating family of sections for a
continuous field. \qed 
\end{lemma}

The continuous fields that we have constructed are not particularly
interesting as continuous fields.  In fact they are isomorphic to
constant fields (they become much more interesting when further
structure is taken into account, as we shall do later in the paper).
For the sequel it will be important to fix a particular isomorphism,
and we conclude this section by doing this.
 
The required unitary isomorphism will be defined using certain cocycle
operators $W_t(C_1,C_2)$, which are analogues of those studied by
Valette in \cite{Valette90} in the case of trees. In the case $q=0$
the cocycle operators for general $\cat(0)$ cube complexes were
constructed in \cite{GuentnerHigson}.  The case where $q>0$ involves
only a minor elaboration of the $q=0$ case, and so we shall refer to
\cite{GuentnerHigson} for details in what follows.

\begin{definition}\label{cocycle0}
If $D$ is a $q$-cube that is adjacent to a hyperplane $H$,
then define $D^{op}$ to be  the opposite face to $D$ in the unique
$(q+1)$-cube that is cut by $H$ and contains $D$ as a $q$-face (such
a cube exists by Lemma \ref{unique-adjacent}).  In the case $D$ is
oriented, we orient $D^{op}$ compatibly.  In either case, we shall
refer to a pair such as $D$ and $D^{op}$ as being adjacent across
$H$. 
\end{definition}

\begin{definition}\label{cocycle}
Let $C$ and $C^{op}$ be adjacent across a hyperplane $H$, as in the
previous definition. If $D$ is any oriented $q$-cube that is adjacent
to $H$, then for $t\in (0,\infty]$ we define 
\[
W_t(C^{op},C)D = 
\begin{cases}
(1-e^{-  t^2})^{1/2} D - e^{-\frac 12 t^2}D^{op}, & \text{if $D$ is separated from $C$ by $H$}\\
e^{-\frac 12 t^2}D^{op} + (1-e^{-  t^2})^{1/2} D, & \text{if $D$ is not separated from $C$ by $H$;}
\end{cases}
\]
in addition we define 
\[
W_t(C^{op},C)D = D \quad \text{if $D$ is not adjacent to $H$}.
\]
We extend $W_t(C^{op},C)$ by linearity to a linear operator on the
spaces of (full and oriented) Julg-Valette $q$-cochains.
\end{definition}

For example 
\[
W_0(C^{op},C) C = C^{op}\quad\text{and} \quad W_0(C^{op},C) C^{op} = - C,
\]
while
\[
W_\infty (C^{op},C)C= C \quad \text{and} \quad  W_\infty (C^{op},C) C^{op} = C^{op},
\]
and indeed $W_\infty(C^{op},C)$ is the identity operator. More
generally,  when restricted to the two-dimensional space spanned by
the ordered basis $(D,D^{op})$ with $D$ adjacent to $H$ but \emph{not}
separated from $C$ by $H$, the operator $W_t(C^{op},C)$ acts as the
unitary matrix 
\[
\begin{bmatrix}
(1-e^{t^2})^{1/2} & - e^{-\frac 12 t^2} \\
  e^{-\frac 12 t^2}& (1-e^{t^2})^{1/2}
\end{bmatrix} .
\]
In particular, $W_t(C^{op},C)$ extends to a unitary operator on the
completed cochain spaces of Definition~\ref{def-analytic-JV} and
subsequent remark.

Let us now assume that two $q$-cubes $C_1$ and $C_2$ are parallel, but
not necessarily adjacent across a hyperplane. It follows from
Theorem~\ref{general-Sageev}  that that there exists a path of
$q$-cubes $E_1, E_2, \dots, E_n$, with $E_1 = C_1$ and $E_n = C_2$,
where each consecutive pair $E_i$, $E_{i+1}$  
consists  of parallel and adjacent $q$-cubes. For all $t\geq 0$ let us define
\begin{equation}
\label{eq-def-cocycle}
W_t(C_1,C_2) = W_t(E_1,E_2)W_t(E_2,E_3)\dots W_t(E_{n-1}, E_n). 
\end{equation}
This notation, which omits mention of the path, is justified by the following  result:

\begin{proposition}
\label{prop-cocycle}
The unitary operator $W_t(C_1, C_2)$ is independent of the path  from $C_1$ to $C_2$.
\end{proposition}
\begin{proof}
Let $\gamma$ and $\gamma'$ be two cube paths connecting cubes $C_1$ and $C_2$. 
As the cubes $C_1$ and $C_2$ are parallel, by Theorem~\ref{general-Sageev} they can be thought of as vertices in the $\cat(0)$ cube complex  created from their parallelism class. The paths $\gamma$ and $\gamma'$ then  give rise to vertex paths in this $\cat(0)$ cube complex with common beginning and end vertices. In this way we reduce the general case of the proposition  to the zero dimensional case, which has been proved in 
\cite[Lemma~3.3]{GuentnerHigson}. 
\end{proof}

In what follows we shall  use the base vertex $P_0$ that was selected
during  the construction of the Julg-Valette complex. 

\begin{definition}
\label{def-Ut0}
Let $t\in (0,\infty]$. For every oriented $q$-cube $D$ let 
\[
U_t D  = W_t(D_0,D)D,
\]
where $D_0$ is the cube nearest to the base vertex $P_0$ in the
parallelism class of $D$ (see Proposition~\ref{prop-nearest-cube}).
Extend $U_t$ by linearity to a linear operator on the spaces of full
and oriented Julg-Valette $q$-cochains; in particular, on oriented
cochains we have
\begin{equation*}
U_t: \C[X^q]\longrightarrow \C[X^q].
\end{equation*}
\end{definition}

\begin{lemma}\label{Ut-isomorphism}
The linear operator  $U_t$  is a vector space isomorphism.
\end{lemma}

\begin{proof}
  Consider the increasing filtration of the cochain space, indexed by
  the natural numbers, in which the $n$th space is spanned by those
  cubes whose nearest vertex to $P_0$ in the edge-path metric is of
  distance $n$ or less from $P_0$.  The operator $U_t$ preserves this
  filtration. In fact, a simple direct calculation (see \cite[Lemma
  4.7]{GuentnerHigson}) shows that
\begin{equation*}
\begin{aligned}
{U}_t D & = W_t (D_0,D)D \\
& = \text{constant}\cdot D + \text{linear combination of cubes  closer to $P_0$ than $D$.}
\end{aligned}
\end{equation*}
This formula shows that the induced map on associated  graded spaces
is an isomorphism.  So $U_t$ is an isomorphism. 
\end{proof}

\begin{lemma}\label{unitary-isomorphism} 
If  $D_1$ and $ D_2$ are any two oriented $q$-cubes in $X$, then
\[
\langle U_t D_1 , U_t D_2\rangle = \langle  D_1 ,   D_2\rangle_t,
\]
where the inner product on the left hand side is that of $\ell^2(X^q)$.
\end{lemma}

\begin{remark}
The lemma implies that the sesqui-linear form  
$\langle \,\,\, ,\,\,\rangle_t$ is positive definite for each $t > 0$,
since  $\langle \,\,\, , \,\,\rangle$ is positive-definite and $U_t$
is an isomorphism.  
\end{remark}

\begin{proof}[Proof of the lemma]
  We can assume that the $q$-cubes $D_1$ and $D_2$ are parallel and
  compatibly oriented since otherwise both sides of the formula are
  zero.  Let $D_0$ denote the $q$-cube in the parallelism class that
  is nearest to the base vertex $P_0$.  Then the unitarity of $W_t$
  and Proposition~\ref{prop-cocycle} give
\[
\begin{aligned}
\langle U_t D_1, U_t D_2\rangle & = \langle W_t(D_0,D_1)D_1, W_t(D_0,D_2)D_2\rangle \\
    & = 
\langle W_t(D_0,D_2)^* W_t(D_0,D_1)D_1, D_2\rangle \\
& = \langle W_t(D_2, D_0) W_t(D_0,D_1)D_1, D_2\rangle  \\
& = \langle W_t(D_2,D_1)D_1, D_2\rangle .
\end{aligned}
\]
But, by an elaboration of \cite[Proposition~3.6]{GuentnerHigson} we have
\begin{equation}
\label{eq-cocycle-calc}
W_t(D_2,D_1)D_1 = e^{-\tfrac 12 t^2 d(D_2,D_1)}D_2 + \;  
    \text{multiples of oriented cubes  other than $D_2$}.
\end{equation}
 Hence we conclude that
\[
\langle W_t(D_2,D_1)D_1, D_2\rangle  = 
   \tfrac{1}{2}e^{-\tfrac 12 t^2 d(D_2,D_1)} = \langle D_1,D_2\rangle_t ,
\]
as required. 
\end{proof}

The following results are immediate consequences of the above:

\begin{theorem}
For all $t\in (0,\infty]$ the map 
\[
U_t\colon \C[X^q] \longrightarrow \C[X^q]
\]
extends to a unitary isomorphism 
\begin{equation*}
\pushQED{\qed} 
U_t: \ell_t^2 (X^q)\longrightarrow \ell^2_\infty(X^q).
\qedhere
\popQED
\end{equation*}
\end{theorem}

\begin{theorem}
The unitary operators $U_t$ determine a unitary isomorphism from the continuous field
 $\{\ell_t^2(X^q)\} _{t\in (0,\infty]}$   generated by sections in Lemmas \ref{lemma-cont-fld1} and \ref{lemma-cont-fld2}
 to the constant field with fiber $\ell^2(X^q)$. \qed
  \end{theorem}


\section{Extension of the Continuous Field}
\label{sec-extension-field}

In this section we shall extend the continuous fields over
$(0,\infty]$ defined in Section~\ref{sec-cont-field} by adding the following  fibers
at $t=0$.  

 \begin{definition}
\label{def-analytic-PS}
We shall denote by $\ell^2_0 (X^q)$ the completion of the space of
oriented Pytlik-Szwarc $q$-cochains in the inner product of
Definition~\ref{def-ps-inner-product}.  It is the subspace of
anti-symmetric functions in the Hilbert space of all square-summable
functions on the set of oriented symbols $\H_q$.
\end{definition}

The following two definitions focus on the particular continuous sections that we shall extend.

\begin{definition} 
\label{def-basic-cochain}
Let $p,q\ge 0$ and let  $(C,D)$ be  an oriented $(p,q)$-cube pair.
The associated  \emph{basic $q$-cochain of type $p$} is the linear 
combination
\[
 f_{C, D} = \sum_{E \parallel_C D} (-1)^{d(D,E)}  E
\]
in the full cochain space.  Here, the sum is over those $q$-cubes $E$
in $C$ that are parallel to $D$, each of which is given the
orientation compatible with the orientation of $D$.  The associated
basic \emph{oriented} cochain is
\begin{equation*}
  f_{\langle C,D \rangle} = f_{C,D}-f_{C,D^*} =
     \sum_{E \parallel_C D} (-1)^{d(D,E)}  \langle\, E \,\rangle,
\end{equation*}
belonging to the space $\C[X^q]$ of oriented $q$-cochains.
\end{definition}

\begin{example}  
  For $q\ge 0$, a basic $q$-cochain of type $p=0$ is just a single
  oriented $q$-cube. A basic $0$-cochain of type $1$ is a
  difference of vertices across an edge. Finally, if $p+q > \dim (X)$
  then there are no basic $q$-cochains of type $p$, since there are no
  $(p+q)$-cubes in $X$.
\end{example}

\begin{definition}
\label{def-basic-sec}
A \emph{basic section of type $p$} of the continuous field 
$\{\, \ell^2_t (X^q) \,\}_{t\in (0,\infty]}$ is a continuous
section of the form 
\[
(0,\infty] \ni  t\longmapsto  t^{-p} f_{\langle C,D \rangle}\in \ell^2 _t (X^q),
\]
where $(C,D)$ is an oriented $(p,q)$-cube pair.
\end{definition}

We shall extend the basic sections to sections over $[0,\infty]$ by assigning to each of them a value at $t=0$ in the Hilbert space $\ell^2 _0(X^q)$, namely the  Pytlik-Szwarc symbol associated to the cube pair $(C,D)$, as in Section~\ref{sec-eq-complex}.  We shall write it as 
\[
\langle\, C,D \,\rangle = [C,D] - [C,D]^*\in \ell^2 _0(X^q) .
\]
Compare Definition~\ref{def-symbol-orientation} and Remark~\ref{ps-cochain-remark}.
We shall prove the following result.

\begin{theorem}
\label{thm-basic-sections}
Let $q\ge 0$.
 \begin{alist}
 \item The pointwise inner product 
 \[
 \bigl \langle t^{-p_1}f_{C_1,D_1}, t^{-p_2}f_{C_2,D_2}\bigr \rangle_t
 \]
  of any two basic sections \textup{(}of possibly different
  types\textup{)} extends to a continuous function on $[0,\infty]$. 
 
 \item The value of this continuous function at $0\in [0,\infty]$ is
   equal to the inner product  
 \[
\bigl \langle [C_1,D_1], [C_2,D_2] \bigr \rangle _0.
 \]
\end{alist}
\end{theorem}

\begin{example}\label{example-tree}
  Suppose that $X$ is a tree.  When $q=1$, the only basic sections are
  those of type $p=0$, and they are the functions $t\mapsto E$, where
  $E$ is an oriented edge in $X$.  Theorem~\ref{thm-basic-sections} is
  easily checked in this case. When $q=0$ there are basic sections
  $t\mapsto Q$ of type $p=0$, which are again easily handled, but also
  basic sections of type $p=1$.  These have the form
\[
t \mapsto t^{-1}( P - Q),
\]
where $P$ and $Q$ are adjacent vertices in the tree.  One calculates
that 
\[
\bigl \langle  t^{-1}( P - Q), t^{-1}( P - Q)\bigr \rangle_t = 
        2 t^{-2}(1-e^{-\frac 12 t^2}) ,
\]
which converges to $1$ as $t\to 0$, in agreement with
Theorem~\ref{thm-basic-sections}.  In addition if $t^{-1}(R-S)$ is a
second, distinct basic cochain, and if the vertices $P,Q,R,S$ are
arranged in sequence along a path in the tree, then a short
calculation reveals that if $d$ is the distance between $Q$ and $R$,
then
\[
\bigl \langle  t^{-1}( P - Q), t^{-1}( R - S)\bigr \rangle_t =  
     -t^{-2} e^{-d \frac 12 t^2}(1-e^{-\frac 12 t^2})^2 = O(t^2) .
\]
In particular the inner product converges to $0$ as $t\searrow 0$, again
in agreement with Theorem~\ref{thm-basic-sections}.   

\end{example}

\begin{definition}
\label{def-extended-basic-sec}
An \emph{extended basic section of type $p$} of the continuous field of Hilbert spaces
$\{\, \ell^2_t (X^q) \,\}_{t\in [0,\infty]}$ is a section of the
form
\begin{equation*}
    t \longmapsto  \begin{cases}
              \langle C,D \rangle, & t=0 \\
              t^{-p} f_{\langle C,D \rangle}, & t>0,
  \end{cases}
\end{equation*}
where $(C,D)$ is an oriented $(p,q)$-cube pair.
\end{definition}

The basic sections form a generating family of sections for the
continuous field $\{ \ell ^2 _t (X^q) \}_{t\in (0,\infty]}$, and of
course the symbols $\langle C,D\rangle$ span $\ell_0^2 (X^q)$.  So it
follows from the theorem that the extended basic sections form a
generating family of sections for a continuous field over $[0,\infty]$
with fibers $\ell^2_t (X^q)$, whose restriction to $(0,\infty]$ is the
continuous field of the previous section.

We shall   prove Theorem~\ref{thm-basic-sections} by carrying out a
sequence of smaller calculations.  The following formula is common to
all of them, and it will also be of use in
Section~\ref{sec-field-ps-ops}.  Here, and subsequently, we shall
write $O(t^p)$ for any finite sum of oriented $q$-cubes times
coefficient functions, each of which is bounded by a constant times
$t^{p}$ as $t\searrow 0$.

\begin{lemma}
\label{lem-7.1-target}
  If $(C,D)$ is an oriented $(p,q)$-cube pair then
  \begin{equation}
\label{eq-7.1-target}
     \sum_{E\parallel_CD} (-1)^{d(D,E)} W_t(D,E)E 
      =  (-t)^p D^{\mathrm op} + O(t^{p+1}),
  \end{equation}
where $D^{\mathrm op}$ is the $q$-face of $C$ separated from $D$ by
the complementary hyperplanes of the pair $(C,D)$, with compatible
orientation.  
\end{lemma}
\begin{proof}
We shall prove the lemma by induction on $p$.
The case $p=0$ is clear.  As for the case $p>0$, let $H$ be a
hyperplane that cuts $C$ but not $D$.  Our aim is to apply the
induction hypothesis to the codimension-one faces of $C$ separated by
$H$.  Denote these faces by $C_\pm$ with $C_+$ being the face
containing $D$; denote $D_+=D$ and $D_-$ the face of $C_-$ directly
across $H$ from $D$; and finally denote by $D_\pm^{\mathrm op}$ the
face in $C_\pm$ separated from $D_\pm$ by all the complementary
hyperplanes of the pair $(C,D)$ \emph{except} $H$.  We have, in
particular, $D^{\mathrm op} = D_-^{\mathrm op}$.

Now, the expression on the left hand side of (\ref{eq-7.1-target})
depends on the cube pair $(C,D)$ and for the course of the proof we
shall denote it by $g_{C,D}$.  We compute the summand of $g_{C,D}$
corresponding to a face $E$ that belongs to $C_-$ using the path from
$D_+$ to $D_-$ and on to $E$.  Doing so, we see that
\begin{align*}
  g_{C,D} &= g_{C_+,D_+} - W_t(D_+,D_-) g_{C_-,D_-} \\
         &= (1-e^{-\frac 12 t^2}) \, g_{C_+,D_+} - 
                 (1-e^{-t^2})^{\frac 12} \, g_{C_-,D_-}.
\end{align*}
Here, we have used that the coefficient of $g_{C_+,D_+}$ at a face $E$
of $C_+$ equals the coefficient of $g_{C_-,D_-}$ at the face of $C_-$
which is directly across $H$ from $E$.  By the induction hypothesis,
$g_{C_+,D_+} = (-t)^{p-1}D_+^{\mathrm op}+O(t^p)$, which is
$O(t^{p-1})$.  Since $1-e^{-\frac 12 t^2}$ is $O(t^2)$ the first term
in this expression is $O(t^{p+1})$.  As for the second term, again by
induction we have $g_{C_-,D_-} = (-t)^{p-1}D_-^{\mathrm op} + O(t^p)$,
which is $O(t^{p-1})$.  It follows that
\begin{align*}
   - (1-e^{-t^2})^{\frac 12} \, g_{C_-,D_-} &= 
        -t \, g_{C_-,D_-} + (t- (1-e^{-t^2})^{\frac 12}) \, g_{C_-,D_-} \\
     &= (-t)^p D_-^{\mathrm op} + O(t^{p+1}) +
                (t- (1-e^{-t^2})^{\frac 12}) O(t^{p-1}) \\
     &=  (-t)^p D^{\mathrm op} + O(t^{p+1}),
\end{align*}
where we have used that $t- (1-e^{-t^2})^{\frac 12}$ is
$O(t^3)$. Putting things together, the lemma is proved.
\end{proof}

In the previous section we defined unitary isomorphisms
$U_t:\ell^2_t(X^q)\to \ell^2(X^q)$.  While these were defined using a specific 
choice of base point within each parallelism class of $q$-cubes, the
choice is not important as far as the unitarity of $U_t$ is concerned.  We shall exploit this by making judicious choices of
base point to calculate the inner products in
Theorem~\ref{thm-basic-sections}.

\begin{lemma}
\label{lemma-equiv-basic1}
Let $(C,D)$ be an oriented $(p,q)$-cube pair, and let $f_{C,D}$ be the
associated basic $q$-cochain of type $p$. The pointwise inner product
\[
\bigl\langle t^{-p}f_{C,D},t^{-p}f_{C,D}\bigr\rangle_t 
\]
converges to $\tfrac{1}{2}$ as $t\searrow 0$.  
\end{lemma}
\begin{proof}
  Choose $D$ as the base point for defining the unitary isomorphisms
  $U_t$.  Then $U_t f_{C,D}$ is exactly the expression
  (\ref{eq-7.1-target}) in the previous lemma.  It follows from
  the lemma that 
  \begin{align*}
    \bigl\langle t^{-p}f_{C,D},t^{-p}f_{C,D}\bigr\rangle_t &= 
        \bigl\langle t^{-p}U_t f_{C,D},  t^{-p}U_t f_{C,D} \bigr\rangle_\infty\\
        &= \bigl\langle (-1)^p D^{\mathrm op} + O(t), 
               (-1)^p D^{\mathrm op} + O(t) \bigr\rangle_\infty \\
        &= \tfrac{1}{2}+O(t),
  \end{align*}
and the result follows.
\end{proof}

\begin{lemma}
\label{lemma-equiv-basic2}
Let $(C_1,D_1)$ and $(C_2,D_2)$ be parallel $(p,q)$-cube
pairs of the same parity, in which the $q$-dimensional faces are
compatibly oriented.  The pointwise inner product
\[
\bigl\langle t^{-p}f_{C_1,D_1},t^{-p}f_{C_2,D_2}\bigr\rangle_t 
\]
converges to $\tfrac{1}{2}$ as $t\searrow 0$.
\end{lemma}
\begin{proof}
  We may assume that $D_2$ lies on the same side of each of the
  complementary hyperplanes of the parallelism class as $D_1$; indeed
  replacing $D_2$ by this face, if necessary, does not change the
  corresponding basic cochain.  Choose $D_1$ as the base point for
  defining the unitary isomorphisms $U_t$, so that by
  Lemma~\ref{lem-7.1-target} we have
  \begin{equation*}
    U_t f_{C_1,D_1} = (-t)^p D_1^{\mathrm op} + O(t^{p+1})
  \end{equation*}
and also, using the identity $W_t(D_1,E) = W_t(D_1,D_2)W_t(D_2,E)$ for
the $q$-dimensional faces $E$ of $C_2$, 
\begin{equation*}
    U_t f_{C_2,D_2} = (-t)^p W_t(D_1,D_2) D_2^{\mathrm op} + O(t^{p+1}).
\end{equation*}
But, the hyperplanes separating $D_1$ and $D_2$ are precisely those
separating $D_1^{\mathrm op}$ and $D_2^{\mathrm op}$, so that 
by (\ref{eq-cocycle-calc}) we have
\begin{align*}
  W_t(D_1,D_2) D_2^{\mathrm op} &= 
        W_t(D_1^{\mathrm op},D_2^{\mathrm op}) D_2^{\mathrm op} \\ 
       &= e^{-\frac{1}{2}d(D_1,D_2) t^2} D_1^{\mathrm op} + 
                \text{terms orthogonal to $D_1^{\mathrm op}$}. 
\end{align*}
Putting everything together we get
\begin{align*}
  \bigl\langle t^{-p}f_{C_1,D_1},t^{-p}f_{C_2,D_2}\bigr\rangle_t  &=
      \bigl\langle t^{-p}U_tf_{C_1,D_1},t^{-p}U_tf_{C_2,D_2}\bigr\rangle_\infty \\ 
          &= e^{-\frac{1}{2}d(D_1,D_2) t^2} 
  \bigl\langle (-1)^p D_1^{\mathrm op}, (-1)^p D_1^{\mathrm op} \bigr\rangle_\infty + O(t)
\end{align*}
and the result follows from this.
\end{proof}

\begin{lemma}
\label{lemma-inequiv-basic}
Let $(C_1,D_1)$ and $(C_2,D_2)$ be oriented cube pairs of types
$(p_1,q)$ and $(p_2,q)$, respectively, and let $f_{C_1,D_1}$ and
$f_{C_2,D_2}$ be the associated basic $q$-cochains.  If $(C_1,D_1)$
and $(C_2,D_2)$ are not parallel, or if $D_1$ and $D_2$ are not
compatibly oriented, then the pointwise inner product
\[
\bigl\langle t^{-p_1}f_{C_1,D_1},t^{-p_2}f_{C_2,D_2}\bigr\rangle_t 
\]
converges to $0$ as $t\searrow 0$.  In particular, this is the case if
$p_1\neq p_2$.
\end{lemma}
\begin{proof}
  If $D_1$ and $D_2$ fail to be parallel or have incompatible
  orientations, then $f_{C_1,D_1}$ and $f_{C_2,D_2}$ are orthogonal in
  the full cochain space for all $t>0$, and the lemma is proved.  So
  we can assume that $D_1$ and $D_2$ \emph{are} parallel and
  compatibly oriented, and therefore that $C_1$ and $C_2$ are not
  parallel.  There is then, after reindexing if necessary, a
  hyperplane $H$ that passes through $C_2$ but not $C_1$, and through
  neither $D_1$ nor $D_2$.  Choose as a base point for the unitary
  $U_t$ a $q$-dimensional face $D$ of $C_2$ which is parallel to the
  $D_i$, compatibly oriented, and on the same side of $H$ as the cube
  $C_1$.  So $f_{C_2,D_2}=\pm f_{C_2,D}$ and also
  \begin{align*}
    \bigl\langle t^{-p_1}f_{C_1,D_1},t^{-p_2}f_{C_2,D}\bigr\rangle_t &=
    \bigl\langle t^{-p_1}U_t f_{C_1,D_1},t^{-p_2}U_t f_{C_2,D}\bigr\rangle_\infty\\
    &= \bigl\langle t^{-p_1}U_t f_{C_1,D_1},(-1)^{p_2} D^{\mathrm op} + O(t) \bigr\rangle_\infty,
  \end{align*}
  where $D^{\mathrm op}$ is the face of $C_2$ separated from $D$ by
  all the complementary hyperplanes of the pair $(C_2,D)$.  In
  particular, $D$ and $D^{\mathrm op}$ are on opposite sides of $H$.
  Now, it follows from the definition of $U_t$ and basic properties of
  the cocycle $W_t$ that all cubes
  appearing in the support of $U_t f_{C_1,D_1}$ are on the same side
  of $H$ as $D$.  Further, from Lemma~\ref{lem-7.1-target} we have
  that  $U_t f_{C_1,D_1}$ is $O(t^{p_1})$, so that the inner product
  above is $O(t)$.  
\end{proof}

\begin{proof}[Proof of Theorem~\ref{thm-basic-sections}]
  The possible values of the inner product in (b) are $0$ and $\pm
  1/2$: the positive value occurs when the oriented cube pairs
  $(C_1,D_1)$ and $(C_2,D_2)$ are parallel and aligned; the negative
  value occurs when they are parallel and not aligned; and $0$ occurs
  when they are not parallel.  The result now follows from
  Lemmas~\ref{lemma-equiv-basic2} and \ref{lemma-inequiv-basic}.
\end{proof}

As we already pointed out, Theorem~\ref{thm-basic-sections} allows us
to extend our continuous field to $[0,\infty]$. In the sequel it will
be convenient to work with the following generating family of
continuous bounded sections.  
\begin{definition}
\label{def-bdd-cts-sec}
A (not necessarily continuous) section $\sigma$ of the continuous
field $\{ \ell^2 _t (X^q)\} _{t\in [0,\infty]}$ is \emph{geometrically
  bounded} if there is a finite set $A\subseteq X_q$ such that $s(t)$
is supported in $A$ for all $t\in (0,\infty]$.
\end{definition}

\begin{proposition}
\label{prop-bdd-cts-sec}
The space of geometrically bounded, continuous sections of  the
continuous field $\{ \ell^2 _t (X^q)\} _{t\in [0,\infty]}$ is spanned
over $C[0,\infty]$ by the extended basic continuous sections.
\end{proposition}

\begin{proof}
Every basic continuous section is certainly geometrically bounded.  If
$X$ is a finite complex, then the converse is true since the fiber
dimension of the continuous field is finite and constant in this case,
and   so the continuous field is a vector bundle, while   the basic
continuous sections span each fiber of the bundle.   In the general
case, we can regard any geometrically bounded continuous section as  a
section of the continuous field associated to a suitable finite
subcomplex, as in Lemma~\ref{lem-increasing-union}, and so express it
as a combination of basic continuous sections. 
\end{proof}


\section{Differentials on the  Continuous Field}
\label{sec-field-ps-ops}

The purpose of this section is to construct differentials
\[
\ell_t^2 (X^0) \stackrel{  d_t}\longrightarrow \ell_t^2 (X^1) \stackrel{  d_t}\longrightarrow \cdots 
 \stackrel{  d_t}\longrightarrow  \ell_t^2 (X^{n-1}) \stackrel{  d_t}\longrightarrow   \ell_t^2 (X^{n})
\]
that continuously interpolate between the Julg-Valette differentials
at $t=\infty$  and the Pytlik-Szwarc differentials at $t=0$.   
For later purposes it will be important to use 
\emph{weighted} versions of the Julg-Valette differentials, as in
Definition~\ref{def-weight-fn-and-differential}. But first we shall
proceed without the weights, and then indicate at the end of this
section how the weights are incorporated. 

 Recall that the operators  
\[
U_t\colon  \C[X^q]    \longrightarrow   \C[ X^{q}]
\]
from Definition~\ref{def-Ut0} were proved to be   isomorphisms in
Lemma~\ref{Ut-isomorphism}.  

\begin{definition}  For $t\in (0,\infty]$ we define 
     \begin{equation*}
  d_t = U_t^{-1} d  U_t  \colon  \C[X^q]    \longrightarrow \C[ X^{q+1}], 
\end{equation*}
where  $d$ is the Julg-Valette differential from
Definition~\ref{def-JV-differential}.   In addition, we define  
\[
d_0 \colon  \C[\H_q]\longrightarrow  \C[\H_{q+1}]
\]
to be the Pytlik-Szwarc differential from Definition~\ref{def-ps-differential}.
\end{definition}

We aim to prove the following continuity statement concerning these operators:

\begin{theorem} 
\label{thm-Kailua-redux}
If $\{ \sigma(t) \}$ is any continuous and geometrically bounded section of
the  continuous field    $\{ \ell _t ^2 (X^q) \} _{t\in [0,\infty]}$,
then the pointwise differential   $\{ d_t \sigma(t)\}$  is a continuous and
geometrically bounded section of  $\{ \ell _t ^2 (X^{q+1}) \} _{t\in [0,\infty]}$.  
  \end{theorem}

According to Proposition~\ref{prop-bdd-cts-sec}, the space of continuous and
geometrically bounded sections is  generated as a module
over $C[0,\infty]$ by the extended basic sections, so  it
suffices to prove Theorem~\ref{thm-Kailua-redux} for such a section.
This we shall now do, following two preliminary lemmas.

\begin{lemma}
\label{lemma-new-ut-calc}
Let $(C,D)$ be an oriented $(p,q)$-cube pair and assume that all the
complementary hyperplanes of the pair $(C,D)$ separate $D$ from the
base point $P_0$.  The associated basic $q$-cochain of type $p$
satisfies
\begin{equation*}
  t^{-p} U_t f_{C,D} = W_t(D_0,D_1) D + O(t),
\end{equation*}
where  $D_0$ is the $q$-cube in $X$ that is closest to the base point 
$P_0$ among cubes parallel to $D$, and  $D_1$ is the face of $C$ that is parallel to $D$ and separated from $D$ by all the 
complementary hyperplanes.
\end{lemma}

\begin{proof} According to our definitions, 
$f_{C,D} = (-1)^{p} f_{C,D_1}$ and 
\[
\begin{aligned}
U_t   f_{C,D_1}  
	& = \sum_{E\parallel_CD} (-1)^{d(D_1,E)} W_t(D_0,E)E \\
	& = W_t(D_0,D_1) \sum_{E\parallel_CD} (-1)^{d(D_1,E)} W_t(D_1,E)E\\
        & = (-t)^p D + O(t^{p+1}),
\end{aligned}
\]
where we have applied Lemma~\ref{lem-7.1-target}.  The result
follows.  
\end{proof}

\begin{lemma}
\label{lemma-new-7.2}
Let $(C,D)$ and $D_1$ be as in the previous lemma.  Let $C_0$ be the
nearest cube to $P_0$ in the parallelism class of $C$, let $F$ be the
face of $C_0$ which is parallel to $D$ and separated from the base
point $P_0$ by the complementary hyperplanes, and let $F_1$ be the
face of $C_0$ that is parallel to $D$ and separated from $F$  by the complementary
hyperplanes.  Then
\begin{alist}
  \item $H\wedge F$ is nonzero if and only if $H$ is a complementary
    hyperplane, in which case $H\wedge F\subseteq C_0$;
  \item $d(C_0,C)=d(F,D)=d(F_1,D_1)$;
  \item $W_t(D_0,D_1)D = F + O(t)$.
\end{alist}
\end{lemma}

\begin{proof}
  Consider first the case $q=0$.  In this case, $D_0=P_0$ and the
  vertex $F_1$ is characterized by the following \emph{hyperplane
    property} from the proof of Proposition~\ref{prop-nearest-cube}:
  every hyperplane separating $P_0$ and $F_1$ is parallel to at least
  one determining hyperplane of parallelism class of $C$ (and $C_0$).  

  For (a), $H\wedge F$ is nonzero exactly when $H$ is adjacent to $F$
  and separates it from $P_0$.  The hyperplanes cutting $C_0$
  certainly satisfy this condition.  Conversely, a hyperplane
  satisfying this condition must intersect all determining hyperplanes
  by Lemma~\ref{lem-4-quadrants}, so cannot separate $F_1$ from $P_0$
  and so must cut $C_0$.

  For (b), no determining hyperplane (of the parallelism class of $C$)
  separates $F$ and $D$.  It follows easily that a hyperplane
  separates $C$ and $C_0$ if and only if it separates $F$ and $D$.
  The same argument applies to $F_1$ and $D_1$.

For (c), from the cocycle property we have
\begin{equation*}
  W_t(D_0,D_1)D = W_t(D_0,F_1) W_t(F_1,D_1) D.
\end{equation*}
To evaluate this, observe that a hyperplane appearing along
(a geodesic) path from $F_1$ to $D_1$ must cross every determining
hyperplane.  It follows that $W_t(F,F_1)=W(D,D_1)$ commutes with
$W_t(F_1,D_1)$ and we have
\begin{align*}
  W_t(F_1,D_1) D &= W_t(F,F_1) W(D_1,D)  W_t(F_1,D_1) D \\ 
       &= W_t(F,D) D \\ 
       &= e^{\frac{1}{2}dt^2} F + O(t),
\end{align*}
where $d=d(F,D)$, and the last equality follows from an
elaboration of \cite[Proposition~3.6]{GuentnerHigson}.  Finally, no
hyperplane separating $D_0$ and $F_1$ is adjacent to $F$ so that
$W_t(D_0,F_1)F=F$.  Putting things together, the result follows.

We reduce the general case to the case $q=0$ using
Proposition~\ref{general-Sageev}, according to which the set of
$q$-cubes parallel to $D$ is the vertex set of a $\cat(0)$ cube
complex in such a way that the $(p+q)$-cubes in $X$ correspond to the
$p$-cubes in this complex.  The key observation is that the $p$-cube in
this complex corresponding to the $(p+q)$ cube $C_0$ in the statement
of the lemma is the $p$-cube closest to the vertex corresponding to
$D_0$.
\end{proof}

\begin{proof}[Proof of Theorem~\ref{thm-Kailua-redux}]
  Let $(C,D)$ be an oriented $(p,q)$-cube pair, with associated
  extended basic $q$-cochain
  \begin{equation*}
    \sigma_{C,D}(t) = \begin{cases}
           [C,D], & t=0 \\ t^{-p} f_{C,D}, & t>0 .
                 \end{cases}
  \end{equation*}
We shall show that the section  $\{ d_t\sigma(t) \}_{t\in [0,\infty]}$   is a
linear combination of extended basic cochains, plus a term that  is
geometrically bounded and $O(t)$.   
  
After possibly changing a sign, we can assume that $D$ is the furthest
from the base point among the $q$-dimensional faces of $C$ parallel to
$D$.  In other words, we can assume that the complementary hyperplanes
$H_1,\dots,H_p$ of the pair $(C,D)$ separate $D$ from the base point.
Each $H_i\wedge D$ is therefore a $(q+1)$-dimensional face of $C$, and
we shall show that
  \begin{equation*}
    d_t (\sigma_{C,D}(t)) = \sum_{i=1}^p \sigma_{C,H_i\wedge D}(t)
      + O(t).
  \end{equation*}
We have equality when $t=0$, 
so   it suffices to show that
\begin{equation*}
 d_t (t^{-p} f_{C,D})  =  t^{-(p-1)}  \sum_{i=1}^p f_{C,H_i\wedge D} 
     + O(t) 
\end{equation*}
for $t>0$, or equivalently  that
\begin{equation}
\label{eq-d-t-formula2}
  d\, U_t ( t^{-p} f_{C,D} )  =  
       \sum_{i=1}^p U_t ( t^{-(p-1)} f_{C,H_i\wedge D} ) + O(t) .
\end{equation}  

As for the left hand side of (\ref{eq-d-t-formula2}), applying
Lemmas~\ref{lemma-new-ut-calc} and \ref{lemma-new-7.2} we have
\begin{equation*}
  d\, U_t ( t^{-p} f_{C,D} )  =  d\,  F  + O(t) 
          = \sum_{i=1}^p H_i\wedge F + O(t) ,
\end{equation*}
where $F$ is as in the statement of Lemma~\ref{lemma-new-7.2}.  So, to
complete the verification of  (\ref{eq-d-t-formula2}) it suffices to
check that
\begin{equation*}
  U_t ( t^{-(p-1)} f_{C,H_i\wedge D} ) = H_i\wedge F + O(t).
\end{equation*}
But this follows from  
Lemmas~\ref{lemma-new-ut-calc} and \ref{lemma-new-7.2}, applied to the
$(p-1,q+1)$-cube pair $(C,H_i\wedge D)$ (although a little care must be taken
here since the base cube $D_0$ that is nearest to $P_0$ within the
parallelism class of $D$  should be replaced by an analogous base
cube  for the parallelism class of $H_i\wedge D$).
\end{proof}

Consider now the adjoint operators 
\begin{equation}
\label{eq-deltat-def}
  \delta_t = U^{-1}_t \delta  U_t  \colon \C [X^q] \longrightarrow \C [X^{q+1}] 
\end{equation}
for $t>0$, together with the adjoint Pytlik-Szwarc differential
\[
   \delta_0  \colon \C [\mathcal H^q] \longrightarrow \C [\mathcal H^{q+1}] 
\]

\begin{theorem} 
\label{thm-Kailua-delta-redux}
If $\{\sigma(t)\}$ is any continuous and geometrically bounded section
of the  continuous field    
$\{ \ell _t ^2 (X^{q+1}) \} _{t\in [0,\infty]}$, then $\{ \delta_t \sigma(t)\}$ 
is a continuous and geometrically bounded  section of the continuous
field  $\{ \ell _t ^2 (X^{q}) \} _{t\in [0,\infty]}$.   
  \end{theorem}

\begin{proof}
While this could be approached   through computations similar to those used to prove  Theorem \ref{thm-Kailua-redux}, there is a shortcut.   Each continuous and geometrically bounded section can be viewed as associated to a finite subcomplex of $X$ as in Lemma~\ref{lem-increasing-union}.  In the case of a finite complex the differentials $\{d_t\}$ constitute a map of vector bundles, and their pointwise adjoints $ \{ \delta_t\}$   automatically give a map of vector bundles too.
\end{proof}

Finally, we return to the issue of weights, which will be important in the next section when we work in the context of Kasparov theory.   Let $\weight_t$ the function on hyperplanes defined by the formula
\begin{equation}
\label{eq-t-weight}
\weight_t(H) = \begin{cases}
1+t\operatorname{dist} (H,P_0) , & 0 < t \leq 1\\
1+  \operatorname{dist} (H,P_0)  & 1\leq t \leq \infty. 
\end{cases}
\end{equation}
In the next section we shall work with the weighted operators
     \begin{equation}
\label{eq-t-weight2}
  d_t = U_t^{-1} d_{w_t} U_t  \colon  \C[X^q]    \longrightarrow \C[ X^{q+1}], 
\end{equation}
for $t>0$, where as before
$
U_t 
$
is the isomorphism from Definition~\ref{def-Ut0},   and where $ d_{w_t}$ is   the weighted Julg-Valette differential described in Definition~\ref{def-weight-fn-and-differential}.  
     
If  $t>0$, then operator in \eqref{eq-t-weight2}  does \emph{not} extend from $\C[X^q]$ to a bounded operator between $\ell^2_t$-spaces.    
But since the pointwise values of a  geometrically bounded section lie in $\C[X^q]$,   Theorem~\ref{thm-Kailua-redux}  makes sense in the weighted case without extending the domains of the operators $d_t$ in \eqref{eq-t-weight2} beyond  $\C[X^q]$.  Moreover the theorem remains true for the weighted family of operators. The proof reduces immediately to the unweighted case because the weighted and unweighted differentials, applied to a continuous and  geometrically bounded section, differ by an $O(t)$ term.  The same applies to Theorem~\ref{thm-Kailua-delta-redux}.

 
\section{Equivariant Fredholm Complexes}

We shall  assume from now on that a second countable, locally compact Hausdorff  topological group\footnote{The topological restrictions on the group $G$ are not really necessary, but they will allow us to easily fit  the concept of equivariant Fredholm complex into the context of Kasparov's $KK$-theory in the next section.}  $G$ acts on our $\cat(0)$ cube complex $X$ (preserving the cubical structure). We shall \emph{not} assume that $G$ fixes any base point in $X$.  

Our goal in this section to place the Julg-Valette and Pytlik-Szwarc complexes  within the context of equivariant Fredholm complexes, and we need to  begin with some definitions.

\begin{definition}
\label{def-fredholm-cplx}
A \emph{Fredholm complex} of Hilbert spaces is a bounded complex of Hilbert spaces and bounded operators for which the identity morphism on the complex is chain homotopic, through a chain homotopy consisting of bounded operators, to a morphism consisting of compact Hilbert space operators.
\end{definition}

In other words, a {Fredholm complex} of Hilbert spaces is a complex of the form
\[
\mathfrak{H}^0 \stackrel d \longrightarrow \mathfrak{H}^1\stackrel d  \longrightarrow \cdots \stackrel d \longrightarrow \mathfrak{H}^n ,
\]
with each $\mathfrak{H}^p$ a Hilbert space and each differential a bounded   operator. Moreover there exist bounded operators 
\[
h : \mathfrak{H}^p \longrightarrow \mathfrak{H}^{p-1} \qquad (p = 1,\dots, n)
\]
such that each operator 
\[
d h + hd : \mathfrak{H}^p \longrightarrow \mathfrak{H}^p \qquad (p = 0,\dots, n)
\]
  is a compact perturbation of the identity operator.   
 
 The Fredholm condition implies that the cohomology groups of a
 Fredholm complex are all finite-dimensional, which is the main reason
 for the definition.  But we are interested in  the following concept
 of \emph{equivariant} Fredholm complex, for which the cohomology
 groups are not so relevant.   
 
 \begin{definition}
 \label{def-eq-fredholm-cplx}
 Let $G$ be a second countable Hausdorff locally compact topological group.  A \emph{$G$-equivariant Fredholm complex} of Hilbert spaces is a bounded complex of separable Hilbert spaces and bounded operators for which 
 \begin{alist}
 
 \item Each Hilbert space carries a continuous  unitary representation of $G$.
 
 \item  The differentials $d$ are not necessarily equivariant, but the differences $d - gdg^{-1}$ are compact operator-valued and norm-continuous functions of $g\in G$.
 
 \item The identity morphism on the complex is chain homotopic, through a chain homotopy consisting of bounded operators, to a morphism consisting of compact Hilbert space operators.
 
 \item The operators $h$ in the chain homotopy above are again  not necessarily equivariant, but the differences $h - ghg^{-1}$ are compact operator-valued and norm-continuous functions of $g\in G$.
 \end{alist}
 
 \end{definition}
 
\begin{remark}
Because the differentials are not necessarily equivariant, the cohomology groups of an equivariant Fredholm complex of Hilbert spaces do not necessarily carry actions of $G$, and so are not of direct interest themselves as far as $G$ is concerned.  Nevertheless the above definition, which is due to Kasparov  (in a minor variant form), has played an important role in a number of mathematical areas, most notably the study of the Novikov conjecture in manifold topology \cite{KasparovEqKK}  (see \cite{BCH93} for a survey of other topics). 
\end{remark}

We are going to manufacture equivariant Fredholm complexes from the
Julg-Valette and Pytlik-Szwarc complexes.  The Julg-Valette complex is
the more difficult of the two  to understand.  Disregarding the group
action, the Julg-Valette differentials from
Definition~\ref{def-JV-differential} extend to bounded operators on
the Hilbert space completions of the cochain spaces associated to the
inner products in \eqref{inner-product}, and the resulting complex of
Hilbert spaces and bounded operators is Fredholm, as in
Definition~\ref{def-fredholm-cplx}.  Moreover the group $G$ certainly
acts unitarily. But the Julg-Valette differentials typically fail to
be $G$-equivariant, since they are defined using a choice of base
point  in the complex $X$ which need not be fixed by $G$. This means
that the technical items (b) and (d) in
Definition~\ref{def-eq-fredholm-cplx}  need to be considered
carefully.  

In fact to handle these technical items it will be necessary to
finally make use of the weight functions $w(H)$ that we introduced in
Definition~\ref{def-weight-fn-and-differential}.  The following
computation will be our starting point. Assemble together all the
Julg-Valette cochain spaces so as to form the single space  
\[
\C [X^\bullet ] = \bigoplus_{q=0}^{\dim (X)}   \C [X^q],
\]
and then form the Hilbert space completion 
\[
\ell ^2 (X^\bullet ) = \bigoplus_{q=0}^{\dim (X)}  \ell^2 (X^q).
\]

\begin{lemma}
For any weight function $w(H)$ the Julg-Valette  operator 
\[
D = \dps + \deltaps
 \, : \,  \ell^2  (X^\bullet) \longrightarrow   \ell^2  (X^\bullet)  ,
\]
viewed as  a densely-defined operator with domain $ \C[X^\bullet]$, is essentially self-adjoint.
\end{lemma}

\begin{proof}
The  operator $D$  is formally self-adjoint in the sense that 
\[
 \langle D  f_1,f_2\rangle = \langle f_1 , D  f_2 \rangle
\]
for all $f_1 , f_2 \in \C[ X^\bullet]$.  The essential
self-adjointness of $D $ is a consequence of the fact that the range
of the operator
\[
 I + \Delta   = I+ D  ^2
\]
is dense in $\ell^2 (X^\bullet)$, and this in turn is a consequence
of the fact that the Julg-Valette Laplacian is a diagonal
operator, as indicated in Proposition~\ref{prop-delta-w-diagonal2}. 
\end{proof}

Since $D $ is an essentially self-adjoint operator, we can study the
resolvent operators $(D  \pm  iI)^{-1}$, which extend from their
initial domains of definition (namely the ranges of $(D\pm iI)$ on
$\C[X^\bullet]$)  to bounded operators on $\ell^2 (X^\bullet)$.

\begin{lemma}
\label{lemma-compact-resolvent}
If  $w$ is a  weight function that is \emph{proper} in the sense that for every $d>0$ the set 
$
\{\, H : w(H) < d\,\}$
is finite, then the  resolvent operators
\[
(D \pm  iI)^{-1}
\,\colon\, \ell^2 (X^\bullet) \longrightarrow  \ell^2 (X^\bullet)
\]
are compact Hilbert space operators.
\end{lemma}

\begin{proof}
The two resolvent operators are adjoint to one another, and so it suffices to show that the product
\[
(I + \Delta )^{-1} = (D  +   iI)^{-1}(D  -  iI)^{-1}
\]
is compact.  But the compactness of $(I + \Delta )^{-1} $ is clear
from Proposition~\ref{prop-delta-w-diagonal2}. 
\end{proof}

Let us now examine the dependence of the Julg-Valette operator $D$ on the initial choice of base point in $X$.

\begin{lemma}
\label{lemma-bounded-diff}
If  $w$ is a weight function that is  \emph{$G$-bounded} in the sense that 
\[
\sup_{H}\,  \bigl | w(H) - w(gH) \bigr | < \infty
\]
for every $g\in G$,  then 
\[
\bigl \| D - g(D) \bigr \| < \infty .
\]
That is, the difference $D - g(D)$, which is a linear operator on $\C[X^\bullet]$, extends to a bounded linear operator on $\ell^2 (X^\bullet)$.
\end{lemma}

\begin{proof}
It suffices to prove the estimate for $d$ in place of $D = d + \delta$, since $d$ and $\delta$ are adjoint to one another.
Now 
\[
d C - g(d) C = \sum _H w(H) H \wedge _{P_0} C - \sum _H w(g(H)) H \wedge _{g(P_0)} C,
\]
where $\wedge_{P_0}$ and $\wedge_{g(P_0)}$ denote the  operators of Definition~\ref{definition-wedge} associated to the two indicated choices of base points.  Since $ w(H) - w(gH)$ is uniformly bounded we can replace $w(g(H)) $ by $w(H)$ in the second sum, and change the overall expression only by a term that defines a  bounded operator.  So it suffices to show that for any pair of base points $P_0$ and $P_1$ the expression 
\[
\sum _H w(H) \bigl ( H \wedge _{P_0} C -   H \wedge _{P_1} C\bigr ) 
\]
defines a bounded operator. But the expression in parentheses is only non-zero when $H$ separates $P_0$ from $P_1$, and there are only finitely many such hyperplanes.  So the lemma follows from the fact that for any hyperplane $H$ the formula 
\[
 H \wedge _{P_0} C -   H \wedge _{P_1} C
 \]
 defines a bounded operator, as long as the cube complex $X$ has bounded geometry.
\end{proof}

From now on we shall assume that the Julg-Valette complex is weighted
using a proper and $G$-bounded weight function. In fact, in the next
section we shall work with the specific weight function $w_\infty$ in
\eqref{eq-t-weight}, and so let us do the same here, even though it is
not yet necessary.  Since the weighted Julg-Valette differential is
not bounded, we shall need to make an adjustment to fit the weighted
complex into the framework of Fredholm complexes of Hilbert spaces and
bounded operators.  We do this by forming the \emph{normalized}
differentials
\[
d ' =d (I + \Delta) ^{-\frac 12} \colon \ell^2(X^q) \longrightarrow \ell^2(X^{q+1}) 
\]
(where, strictly speaking, by $d$ in the above formula we mean the closure of $d$ in the sense of unbounded operator theory). The \emph{normalized  Julg-Valette complex} is the complex
\begin{equation}
\label{eq-modified-jv-complex}
\ell^2   (X^0) \stackrel{d'}\longrightarrow    \ell^2   (X^1)  \stackrel{d'} \longrightarrow  \cdots  \stackrel{d'}\longrightarrow   \ell^2   (X^n)  .
\end{equation}
It is indeed a complex  because $d$ and $(I + \Delta)^{-\frac 12 }$ commute with one another, and it is a Fredholm complex because the adjoints $d'{}^*$ constitute a chain homotopy between the identity  and a compact operator-valued cochain map.  In fact 
\[
d ' d'{}^* + d'{}^* d ' =  D^2 (I + D^2 )^{-1} = I - (I + D^2 )^{-1} ,
\]
and $(I+ D^2 )^{-1}$ is compact by Lemma~\ref{lemma-compact-resolvent}.

We shall use the following computation from the functional calculus to show   that the normalized complex is an equivariant Fredholm complex of Hilbert spaces. 

\begin{lemma}[Compare  \cite{BaajJulg}]
\label{lem-integral}
If  $  T$ is  a  positive, self-adjoint  Hilbert space operator that is bounded below by some positive constant, then 
\[
T^{-\frac 12} = \frac 2 \pi \int _0^\infty (\lambda^2  + T )^{-1} \, d\lambda 
\]
The integral converges in the norm topology. \qed
\end{lemma}

\begin{theorem}
\label{thm-jv-is-fredholm}
The  normalized Julg-Valette complex 
\[
 \ell^2 (X^0) \stackrel {d'} \longrightarrow  \ell^2 (X^1)  
        \stackrel {d'} \longrightarrow \cdots
        \stackrel {d'} \longrightarrow  \ell^2 (X^n)
\]
that is defined using the proper and $G$-bounded weight function
$w_\infty$ in \eqref{eq-t-weight}  is an equivariant Fredholm complex.    
\end{theorem}

\begin{proof} 
It suffices to show that the normalized operator 
\[
D ' = D (I + D^2 )^{-1/2} = d' + d'{}^*
\]
has the property that $g(D') - D'$ is a compact operator-valued and norm-continuous function of $g\in G$.
For this we use Lemma~\ref{lem-integral} and  the formula
\[
D (\lambda ^2  + 1 +     D^2  )^{-1}   
= \frac 12\bigl  (  (    D  +   i\mu  )^{-1}+  (     D -  i\mu  )^{-1} \bigr  ),
\]
where $\mu = (\lambda^2 +1)^{1/2}$, to write the difference  
$g(D') - D'$  as a linear combination of two integrals 
\[
  \int _0^\infty 
\Bigl ( (g (    D)   \pm i\mu )^{-1} -  (    D \pm i\mu )^{-1} \Bigr )
  \, d\lambda  .
\]
The integrand is 
\begin{equation}
 \label{eq-split-integrand1}
  ( g(   D)\pm i\mu)^{-1} 
\bigl (   D- g(   D)\bigr ) (    D\pm i\mu)^{-1},
\end{equation}
which is a  norm-continuous, compact operator valued function  of
$\lambda \in [0,\infty)$ whose  norm is $O(\lambda^{-2})$ as 
$\lambda \nearrow \infty$. So the integrals converge to compact operators,
as required. 
\end{proof}

Let us now examine the     Pytlik-Szwarc complex. The inner products on the Pytlik-Szwarc cochain spaces given in Definition~\ref{def-ps-inner-product} are $G$-invariant, and  the Pytlik-Szwarc differentials given in Definition~\ref{def-ps-differential} are bounded and $G$-equivariant, so the story here is much simpler.

\begin{theorem}
\label{thm-p-s-is-fredholm}
The   {Pytlik-Szwarc complex} 
\[
\ell^2 _0 (X^0) \longrightarrow    \ell^2 _0 (X^1)  \longrightarrow  \cdots \longrightarrow   \ell^2 _0 (X^n) 
\]
is an equivariant Fredholm complex.
\end{theorem}

\begin{proof}
It follows from  Proposition~\ref{prop-ps-laplacian} that the  formula 
\[
h = \frac{1}{p+q} \,\delta\, \colon  \C[\H^p_q]  \longrightarrow  \C[\H^{p+1}_{q-1}] 
\]
(we set $h=0$ when $p=q=0$) defines an exactly  $G$-equivariant and bounded chain homotopy between the identity and a compact operator-valued cochain map, namely the orthogonal projection onto $\C[\H^0_0]\cong \C$ in degree zero, and the zero operator in higher degrees.  \end{proof}

To conclude this section we introduce the following notion of (topological, as opposed to chain) homotopy between two equivariant Fredholm complexes.  In the next section we shall construct a homotopy between the Julg-Valette and Pytilik-Szwarc equivariant Fredholm complexes we constructed above using the continuous field of complexes constructed in Section~\ref{sec-field-ps-ops}.

\begin{definition} 
\label{def-homotopic-complexes}
Two equivariant complexes of Hilbert spaces $(\mathfrak{H}^\bullet_0,
d_0)$ and $(\mathfrak{H}^\bullet_1 , d_1)$ are \emph{homotopic} if
there is a bounded complex of continuous fields of Hilbert spaces over
$[0,1]$ and adjointable families of bounded differentials for which  
\begin{alist}
  \item Each continuous field   carries a continuous unitary representation of $G$.

  \item  The differentials $d=\{ d_t \}$ are not necessarily
    equivariant, but the differences $d - gdg^{-1}$ are compact
    operator-valued and norm-continuous functions of $g\in G$. 

  \item The identity morphism on the complex is chain homotopic,
    through a chain homotopy consisting of  adjointable families of
    bounded operators, to a morphism consisting of compact operators
    between continuous fields. 
 
 \item The operators $h=\{h_t\}$ in the homotopy above  are again  not
   necessarily equivariant, but the differences $h - ghg^{-1}$ are
   compact operator-valued and norm-continuous functions of $g\in G$. 
  
 \item The restrictions of the complex to the points $0, 1\in [0,1]$
   are the complexes  $(\mathfrak{H}^\bullet_0, d_0)$ and
   $(\mathfrak{H}^\bullet_1 , d_1)$. 
\end{alist}
\end{definition}

We need to supply definitions for    the operator-theoretic concepts mentioned above. These are usually formulated in the language of Hilbert modules, as for example in \cite{Lance}, but for consistency with the rest of this paper we shall continue to use the language of continuous fields of Hilbert spaces.

\begin{definition}
\label{def-adjointable-family}
 An \emph{adjointable family of operators} (soon we shall contract this to \emph{adjointable operator}) between continuous fields $\{\mathfrak{H}_t\}$ and $\{\mathfrak{H}'_t\}$ over the same compact space $T$ is a family of bounded operators 
\[
A_t \colon \mathfrak{H}_t \longrightarrow \mathfrak{H}'_t
\]
that carries continuous sections to continuous sections, whose adjoint family 
\[
A_t^* \colon \mathfrak{H}_t' \longrightarrow \mathfrak{H}_t
\]
also carries continuous sections to continuous sections.  An adjointable operator is \emph{unitary} if each $A_t$ is unitary. \end{definition}

\begin{definition}
 A representation of $G$ as unitary adjointable operators on a continuous field $\{ \mathfrak{H}_t\}$ is \emph{continuous} if the action  map
\[
G \times \{\, \text{continuous sections}\,  \} \longrightarrow  \{\, \text{continuous sections}\,  \}
\]
is continuous. We place on the space of continuous sections the
topology associated to the norm $\|\sigma \| = \max \|\sigma(t)\|$. 
\end{definition}

\begin{definition}
\label{def-cpt-family}
An adjointable operator  $A = \{ A_t\}$ between continuous fields of
Hilbert spaces over the same compact base space $T$ is \emph{compact}
if it is the norm limit, as a Banach space  operator  
\[
A :  \{\, \text{continuous sections}\,  \} \longrightarrow  \{\, \text{continuous sections}\,  \} ,
\]
of a sequence of linear combinations of operators of the form 
\[
\sigma \longmapsto \langle \sigma_1, \sigma \rangle \sigma_2 ,
\]
where $\sigma_1$ and $\sigma_2$ are continuous sections (of the domain
and range continuous fields, respectively).   The compact operators form a closed,
two-sided ideal in the $C^*$-algebra of all adjointable operators. 
\end{definition}

Here, then, is the theorem that we shall prove in the next section:

\begin{theorem}
\label{thm-main-homotopy-thm}
The equivariant Fredholm complexes obtained from  the Julg-Valette and Pytlik-Szwarc complexes in Theorems~\textup{\ref{thm-jv-is-fredholm}} and \textup{\ref{thm-p-s-is-fredholm}} are homotopic \textup{(}in the sense of Definition~\textup{\ref{def-homotopic-complexes}}\textup{)}.
\end{theorem}

 
\section{K-Amenability}
\label{sec-k-amenable}

The purpose of this section is to prove Theorem~\ref{thm-main-homotopy-thm}.   But before giving the proof, we shall explain the $K$-theoretic relevance of the theorem.
To do so we shall need to use the language of Kasparov's equivariant $KK$-theory \cite{KasparovEqKK}, but we emphasize that  the proof of Theorem~\ref{thm-main-homotopy-thm} will involve only the definitions from the last section our work ealier in the paper. We shall assume familiarity with Kasparov's theory.

 A $G$-equivariant complex of Hilbert spaces, as in Definition~\ref{def-eq-fredholm-cplx}, determines a class in Kasparov's equivariant representation ring 
\[
R(G)  = KK_G ( \C ,\C ),
\]
in such a way that 
\begin{alist}
\item homotopic complexes, as in Definition~\ref{def-homotopic-complexes}, determine the same element,  
\item a complex whose differentials are exactly $G$-equivariant determines the same class as the complex of  cohomology groups (these are finite-dimensional unitary representations of $G$)  with zero differentials, and 
\item a complex with the one-dimensional  trivial representation in degree zero, and no higher-dimensional cochain spaces, determines the multiplicative identity element  $1\in R(G)$.
\end{alist}

\begin{definition}  See \cite[Definition 1.2]{JulgValetteQ_p}.
A second countable and locally compact Hausdorff topological group $G$ is \emph{$K$-amenable} if the multiplicative identity element $1\in R(G)$ is representable by an equivariant Fredholm complex of Hilbert spaces
\[
\mathfrak{H}^0 \longrightarrow \mathfrak{H}^1 \longrightarrow \cdots \longrightarrow \mathfrak{H}^n
\]
 in which the each cochain space $\mathfrak{H}^p$, viewed as a unitary representation of $G$,  is weakly contained in the regular representation of $G$.
\end{definition}

\begin{theorem}[See {\cite[Corollary 3.6]{JulgValetteQ_p}}.]
If $G$ is $K$-amenable, then the natural homomorphism of $C^*$-algebras 
\[
C^*_{\max} (G) \longrightarrow C^*_{\mathrm{red}}(G)
\]
induces an isomorphism of $K$-theory groups
\[
\pushQED{\qed}
K_*(C^*_{\max} (G)) \longrightarrow K_*(C^*_{\mathrm{red}}(G)).
\qedhere
\popQED
\]
\end{theorem}
 
\begin{remarks}
The $C^*$-algebra homomorphism in the theorem is itself an isomorphism if and only if the group $G$ is amenable; this explains the term \emph{$K$-amenable}.  Not every group is $K$-amenable; for example an infinite group with Kazhdan's property T is certainly not $K$-amenable, because the $K$-theory homomorphism is certainly not an isomorphism.
\end{remarks}

After having quickly surveyed this background information, we can state the main result of this section:

\begin{theorem}
\label{thm-k-amenable}
If a second countable and locally compact group   $G$ admits a proper action on a bounded geometry $\cat(0)$ cube complex, then $G$ is $K$-amenable. 
\end{theorem}

 The theorem was proved by   Julg and Valette in \cite{JulgValetteQ_p} in the case where the cube complex is a tree.  They used the Julg-Valette complex, as we have called it, for a tree, and showed that the continuous field of complexes that we have constructed in this paper is a homotopy connecting the Julg-Valette  and Pytlik-Szwarc complexes. We shall do the same in the general case.  The construction of this homotopy proves the theorem in view of the following simple result, whose proof we shall omit.

\begin{lemma}
\label{lemma-weak-containment}
Assume that  a second countable and locally compact group   $G$   acts proper action on a  $\cat(0)$ cube complex. 
The Hilbert spaces in the Julg-Valette complex are weakly contained in the regular representation of $G$.  \qed
\end{lemma}

\begin{remark}
Theorem~\ref{thm-k-amenable} is not new; it was proved by Higson and Kasparov in \cite[Theorem 9.4]{HigsonKasparov} using a very different argument that is both far more general (it applies to a much broader class of groups)  and far less geometric.
\end{remark}

To prove Theorem~\ref{thm-k-amenable}  it therefore suffices to 
prove   Theorem~\ref{thm-main-homotopy-thm}, and this is what we shall now do.

We shall construct the homotopy that the theorem requires by modifying the constructions in Section~\ref{sec-field-ps-ops} in more or less the same way that we modified  the Julg-Valette complex to construct the complex \eqref{eq-modified-jv-complex}.  We shall therefore be applying the functional calculus to the family of operators 
\begin{equation}
\label{eq-D-t-def}
D_t = U_t^* (d_{\weight_t} + \delta_{\weight_t} )U_t \colon \ell^2 _t(X^\bullet)\longrightarrow \ell^2_t(X^\bullet) ,
\end{equation}
where $d_{w_t}$ is the Julg-Valette differential associated to the weight function in \eqref{eq-t-weight}, and of course $\delta_{w_t}$ is the adjoint differential.   
To apply the functional calculus we shall need to know that the family of resolvent operators 
\[
(D_t + i \lambda )^{-1}\colon \ell^2 _t (X^\bullet) \longrightarrow \ell ^2 _t (X^\bullet)
\]
carries continuous sections to continuous sections. 
This is a consequence of the following result:

\begin{proposition}
\label{prop-regular}
Let  $\lambda $ be a nonzero real number. The family of operators 
\[
 \bigl \{ \, (D_t  + i\lambda I) ^{-1} \colon \ell^2 _t (X^\bullet) \to \ell^2 _t (X^\bullet)\, \bigr \}_{t\in [0,\infty]}
 \]
 carries the space of continuous and geometrically bounded sections  to a dense subspace of the space of continuous and geometrically bounded sections in the norm $
\|s \| = \sup_{t\in [0,\infty]} \| s(t)\|_{ \ell^2_t (X^\bullet)}
$.
 \end{proposition}

Actually we shall need a small variation on this proposition:

 \begin{definition}
Denote by $  P =\{ P_t\} $ the operator that is in each fiber the orthogonal projection onto the span of the single basic $q$-cochain $f_{P_0,P_0}$ of type $p=0$ (of course this basic cochain  is just  $P_0$).
\end{definition}

  It follows from the formula for the Julg-Valette Laplacian in Proposition~\ref{prop-delta-w-diagonal} that   the operators $  P_t +   \Delta_t$ are essentially self-adoint  and bounded below by $1$. So we can form the resolvent operators $(D_t  + P_t +  i\lambda I) ^{-1} $ for any  $\lambda\in \R$, including $\lambda = 0$.

 \begin{proposition}
\label{prop-regular2}
Let  $\lambda $ be any real number \textup{(}possibly zero\textup{)}. The family of operators 
\[
 \bigl \{ \, (D_t  + P_t +  i\lambda I) ^{-1} \colon \ell^2 _t (X^\bullet) \to \ell^2 _t (X^\bullet)\, \bigr \}_{t\in [0,\infty]}
 \]
 carries the space of continuous and geometrically bounded sections  to a dense subspace of the space of continuous and geometrically bounded sections.
 \end{proposition}

Both propositions will be proved by examining action of the Laplacians
\begin{equation}
\label{eq-Delta-t-def}
\Delta_t = D_t^2  = U_t^* (d_{w_t} + \delta_{w_t})^2 U_t 
\end{equation}
on continuous and geometrically controlled sections of the field $\{ \ell^2 _t (X^\bullet\} _{t\in [0,\infty]}$.

\begin{proof}[Proof of Propositions~\ref{prop-regular} and \ref{prop-regular2}]
The family of operators $ \{ D_t\} $ maps the space of continuous, geometrically bounded sections into itself, so we can consider the compositions  
\[
 \Delta_t +  \lambda^2 I  = (  D_t + i   \lambda I) (  D_t - i  \lambda I) 
\]
and 
\[
\Delta_t  + P_t + \lambda^2 I = (  D_t + P_t +  i   \lambda I) (  D_t + P_t   - i  \lambda I)  ,
\]
and it suffices to show that the families of  these operators  map   the space continuous and geometrically bounded sections into a dense subspace of itself. 

 Let $f_{C,D} \in \C[X^q]$  be a basic $q$-cochain of type $p$.  Lemmas~\ref{lemma-new-ut-calc} and \ref{lemma-new-7.2} tell us that 
 \[
 t^{-p} U_t  f_{C,D}  =  (-1)^p F  +O(t) ,
 \]
 where the $q$-cube $F$ has the property that there are precisely $p$ hyperplanes adjacent to it that separate it from the base point  $P_0$.  So according to our formula for the Julg-Valette Laplacian in Proposition~\ref{prop-delta-w-diagonal}, 
\[
(d_{w_t} +  \delta_{w_t})^2 U_t  :  t^{-p}f_{C,D} \longmapsto  (p+q)\cdot  (-1)^p F  +O(t)
\]
and so, by applying $U_t^*$ to both sides we get 
\[
\bigl ( \Delta_t ^2 + \lambda ^2 I \bigr ) :  t^{-p}f_{C,D} \longmapsto  \bigl ((p+q)+ \lambda ^2\bigr   ) \cdot  t^{-p}  f_{C,D} +O(t) .
\]
Similarly
\[
\bigl ( \Delta_{t}{+}P_t{+} \lambda ^2 I\bigr )  : t^{-p}f_{C,D}\longmapsto \bigl  ( \max\{1,(p+q) \} + \lambda ^2\bigr  ) \cdot t^{-p}   f_{C,D} +O(t) .
\]
So the ranges of the families $\{ \Delta_t + \lambda^2 I\}$ and $\{ \Delta_t + P_t + \lambda^2 I\}$ contain $O(t)$ perturbations of every basic section.  The propositions follow from this.
\end{proof}

 Now form  the  bounded self-adjoint  operators  
\[
  F_t =   D_t (   P_t +   D_t^2 )^{-\frac 12} .
\]
By the above and Lemma~\ref{lem-integral} the family $\{ F_t\} _{t\in [0,\infty]}$ maps continuous sections to continuous sections.  So we can consider  the  bounded complex of continuous fields of Hilbert spaces over $[0,1]$ and bounded adjointable operators 
 \begin{equation}
 \label{eq-field-fred-cplxes}
 \xymatrix@C=40pt{
 \{ \ell^2 _t (X^0) \} _{t\in [0,\infty]}
 	 \ar[r]^-{\{d'_t\}_{t\in [0,\infty]}} &
		 \{ \ell^2 _t (X^1) \} _{t\in [0,\infty]}
 			\ar[r]^-{\{d'_t\}_{t\in [0,\infty]}} & 
				\cdots 
					\ar[r]^-{\{d'_t\}_{t\in [0,\infty]}} &
							\{ \ell^2 _t (X^n) \} _{t\in [0,\infty]} 	
}
 \end{equation}
in which each differential $\{d'_t\}$ is the component of $\{F_t\}$ mapping between the indicated continuous fields.

\begin{proposition}
Disregarding the $G$-action,
the complex \eqref{eq-field-fred-cplxes} is a homotopy of Fredholm complexes.
\end{proposition}

\begin{proof}
If we set
$
h_t  =d'_t{}^*$, then 
\[
h_td_t ' + d_t ' h_t  = \Delta_t  (P_t+\Delta_t)^{-1} = I -P_t (  P_t +  \Delta _t)^{-1} ,
\]
and   $\{ P_t (  P_t \!+\!  \Delta _t)^{-1} \}$,  is    compact operator on the continuous field $\{ \ell^2 _t (X^\bullet)\} _{t\in [0,\infty]}$.
\end{proof}

It remains show that  \eqref{eq-field-fred-cplxes} is an \emph{equivariant} homotopy. 
If the resolvent families $\{ (D_t + P_t + i \lambda I)^{-1} \}$ were compact, then we would be able to follow the route taken in the previous section to prove equivariance of the Fredholm complex associated to the Julg-Valette complex.  But compactness fails at $t=0$, and so we need to be a bit more careful. 
 The following two propositions will substitute for the Lemmas \ref{lemma-compact-resolvent} and \ref{lemma-bounded-diff} that were used to handle the Julg-Valette complex in the previous section.

\begin{proposition}
\label{prop-compact-resolvent-technical}
For every $\varepsilon > 0$ and for every $\lambda \in \R$ the restricted family of operators 
\[
\{ (    D_t  +   P_t \pm i\lambda )^{-1}\} _{t\in [\varepsilon, \infty]}
\]
is a compact operator on the continuous field  $\{ \ell^2 _t (X^\bullet)\} _{t\in [\varepsilon, \infty]}$.  Moreover
\[
  \| ( D_t  +   P_t \pm i\lambda)^{-1}\| \le | 1 + i\lambda |^{-1}
\]
for all $t $ and all $\lambda$. 
\end{proposition}

\begin{proposition}
\label{prop-eq-technical}
For every $g\in G$ the  operators $  D_t - g(  D_t)$ are uniformly  bounded in $t$:
\[
\sup _{t\in [0,\infty]} \|  D_t - g(  D_t) \| < \infty 
\]
Moreover
\[
  \|  D_t - g(  D_t) \|  = O(t).
\]
as $t\to 0$.
\end{proposition}

Taking these for granted, for a moment, here is the result of the calculation:

\begin{theorem}
The complex   \eqref{eq-field-fred-cplxes} is a homotopy of equivariant Fredholm complexes in the sense of Definition~\textup{\ref{def-homotopic-complexes}}.\end{theorem}

\begin{proof} 
 We need to check that the families of differentials  $\{ d'_t\}$ in the complex \eqref{eq-field-fred-cplxes} are $G$-equivariant modulo compact operators, and also that $\{g(d'_t)\}$ varies norm-continuously with $g\in G$.  
 
 Let us discuss norm-continuity first. If $g$ is sufficiently close to the identity in $G$, then $g$ fixes the base point  $P_0$, and for such $g$ we have $g(d'_t) = d'_t$ for all $t$.  So  $\{ g(d'_t)\}$ is actually locally constant as a function of $g$.
 
 The proof of equivariance modulo compact operators is a small variation of the proof of Theorem~\ref{thm-jv-is-fredholm}. It suffices to show that the family of operators
$
 \{ g(F_t) - F_t \} 
 $
is compact.

Since 
\[
\begin{aligned}
  F_t 
  	& =    D _t (   P_t +   \Delta_t )^{-\frac 12}  \\
	& =  (   P_t +   D_t ) (   P_t  +  \Delta_t )^{-\frac 12}+ \text{compact operator},
\end{aligned}
\]
it suffices to prove that the operator  
\[
  E_t  = (   P_t +   D_t)
(   P_t +   \Delta_t)^{-\frac 12} \]
is equivariant modulo compact operators.  
Applying Lemma ~\ref{lem-integral}  we find that 
\[
  \begin{aligned}
  E_t
	& = \frac 2 \pi \int _0^\infty  (   P_t +   D_t)(\lambda^2  I  +    P_t +     \Delta_t  )^{-1}  \, d\lambda  \\
	& =  \frac 1 \pi \int _0^\infty  \bigl ( (    D_t  +   P_t - i\lambda )^{-1}+   (     D_t +   P_t +  i\lambda )^{-1} \bigr )   \, d\lambda  
\end{aligned}
\]
So the difference  $g(   E_t) -   E_t$  is the sum of the two integrals
\begin{equation}
\label{eq-pm-integrals}
 \frac 1 \pi \int _0^\infty 
\Bigl (
 (g (    D_t ) +   g(P_t) \pm i\lambda )^{-1} -  (    D_t  +   P_t \pm i\lambda )^{-1} 
 \Bigr )
  \, d\lambda   
\end{equation}
Now the integrands in \eqref{eq-pm-integrals}  can be written as  
 \begin{multline}
 \label{eq-split-integrand2}
 (g(   P_t)+g(   D_t)\pm i\lambda)^{-1} 
 \bigl (   D_t- g(   D_t)\bigr ) (   P_t+   D_t\pm i\lambda)^{-1} \\
 + 
 (g(   P_t)+g(   D_t)\pm i\lambda)^{-1} 
 \bigl (   P_t- g(   P_t)\bigr ) (   P_t+   D_t\pm i\lambda)^{-1}
\end{multline}
Both terms in \eqref{eq-split-integrand2} are norm-continuous, compact operator valued functions of $\lambda \in [0,\infty)$, the first by virtue of Proposition~\ref{prop-eq-technical} and the second because $P_t$ is compact.  Moreover  the norms of both are $O(\lambda^{-2})$ as $\lambda \to \infty$.   So the integrals in \eqref{eq-pm-integrals} converge to compact operators, as required.
\end{proof}

It remains to prove Propositions~\ref{prop-compact-resolvent-technical} and \ref{prop-eq-technical}.  The first is easy and we can deal with it immediately.  

\begin{proof}[Proof of Proposition~\ref{prop-compact-resolvent-technical}]
We want to show that the family of operators
\[
\{ K_t\} _{t\in [\varepsilon, \infty]}  =\{ (    D_t  +   P_t \pm i\lambda )^{-1}\} _{t\in [\varepsilon, \infty]}
\]
is compact.  Since the compact operators form a closed, two-sided ideal in the $C^*$-algebra of all adjointable families of operators it suffices to show that the family
\[
\{ K_t^*K_t\} _{t\in [\varepsilon, \infty]}  =\{ (    \Delta_t  +   P_t + \lambda^2 )^{-1}\} _{t\in [\varepsilon, \infty]}
\]
is compact; compare  \cite[Proposition 1.4.5]{Pedersen79}.  Conjugating by the unitaries $U_t$  it suffices to prove that the family 
\[
\bigl \{    (d_{w_t}\delta_{w_t}  +\delta_{w_t}   d_{w_t} +   P_t + \lambda^2 )^{-1}\bigr \} _{t\in [\varepsilon, \infty]}
\]
on the constant field of Hilbert spaces with fiber $\ell^2 (X^\bullet)$ is compact; this is one of the things that restricting to $t\in [\varepsilon,\infty]$  makes possible.  But this final assertion is a simple consequence of the explicit formula for the Julg-Valette Laplacian in Proposition~\ref{prop-delta-w-diagonal}, together with the fact that the weight functions $w_t$ are uniformly proper in $t\in [\varepsilon,\infty]$ in the sense that for every $N$, all but finitely many hyperplanes $H$ satisfy $w_t(H)\ge N$ for all $t\in [\varepsilon, \infty]$.

As for the norm estimate in the proposition, this holds not just for $\Delta_t + P_t$ but for any self-adjoint operator bounded below by $1$, and is elementary.
\end{proof}

Let us turn now to Proposition~\ref{prop-eq-technical}.  A complicating factor is that $G$ not only fails to preserve the Julg-Valette differential, but also fails to preserve the unitary operators  $U_t$ that appear in the definitions of the differentials $d_t$.  The proposition is only correct because the two failures to a certain extent cancel one another out.

\begin{definition}
Let $P$ and $Q$ be vertices in $X$.
Define a unitary operator
\[
\widehat W_t (Q,P) \colon \ell^2   (X^q) \longrightarrow \ell^2   (X^q)
\]
as follows.
When $q=0$,   we define $\widehat W_t(Q, P)$ to be  the cocycle operator $W_t(Q,P)$ of Definition~\ref{cocycle}. On higher cubes, $\widehat W_t(Q,P)$ respects the decomposition of $\ell  ^2 (X^q)$ according to parallelism classes, and on a summand determined by a given class we set  $\widehat W_t(Q, P) = W_t(C_{Q}, C_{P})$, where $C_{Q}$ and $C_{P}$ are the cubes in the equivalence class 
nearest to $Q$ and $P$. 
\end{definition} 

It is immediate from the definition of the unitary  operator $U_t$ in Definition~\ref{def-Ut0}     that 
\begin{equation}
\label{eq-g-of-U-t}
g(U_t) = W_t(Q_0, P_0)U_t: \ell^2 _t(X^\bullet)\rightarrow \ell^2(X^\bullet) 
\end{equation}
From this and the definition of $D_t$ we find that 
\begin{equation}
\label{eq-g-of-D-fmla}
g(D_t) = U_t^* \widehat W_t (Q,P) ^*\bigl ( g(d_{\weight_t}) + g(\delta_{\weight_t} )\bigr )\widehat W_t (Q,P)  U_t .
\end{equation}
Now let us use the abbreviation $\widehat W_t:= \widehat W_t (Q,P) $ and write 
 \[
 D_t  -g(D_t)  
 	  =  U_t^* \Bigl ( ( d_{w_t} + \delta_{w_t}) - \widehat W_t^* \bigl (  g(d_{\weight_t}) + g(\delta_{\weight_t} )\bigr )\widehat W_t  \Bigr ) U_t 
	  \]
The right-hand side can be rearranged as 
\[
 U_t^* W_t ^* \Bigl (  W_t  d_{w_t}    -   g(d_{\weight_t} ) W_t    \Bigr ) U_t
	+
	U_t^*   \Bigl (   \delta_{w_t}W_t^*    - W_t ^* g(\delta_{\weight_t} )     \Bigr ) W_tU_t
 \]
and the norm of this expression  is no more than 
\[
\bigl \|   \widehat W_t  d_{w_t}    -   g(d_{\weight_t} ) \widehat W_t   \bigr \| 
	+
	\bigl \|  \delta_{w_t}  \widehat W_t ^* - \widehat W_t ^* g(\delta_{\weight_t} ) \bigr \| 
	\]
So it suffices to show 
  that the operators
 \begin{equation}
 \label{eq-equivariance-check1}
  \widehat W_t  d_{w_t}    -   g(d_{\weight_t} )\widehat W_t
  \quad \text{and} \quad 
 \delta_{w_t}  \widehat W_t ^*   - \widehat W_t ^* g(\delta_{\weight_t} )
 \end{equation}
 satisfy the conclusions of Proposition~\ref{prop-eq-technical}. 
The second operator is   adjoint to the first.  So in fact  it suffices to prove the conclusions of  Proposition~\ref{prop-eq-technical} for the first operator alone.  This is what we shall do.

Before we proceed, let us adjust  our notation a bit, as follows. Given a vertex $P$ in $X$, we shall denote by $d_{P,w_t}$ the Julg-Valette differential that is defined using the base vertex $P$ and the weight function  \eqref{eq-t-weight}, for whose definition we also use the base vertex $P$ rather than $P_0$.  With this new notation we can drop further mention of the group $G$: Proposition~\ref{prop-eq-technical} is a consequence of the following assertion:

\begin{proposition}
\label{prop-eq-technical2}
The operator
\begin{equation*}
   \widehat W_t (Q,P) d_{P,w_t}    -   d_{Q,w_t} \widehat  W_t(Q,P) \colon \C[X^q] \longrightarrow \C[X^{q+1}]
\end{equation*}
is bounded for all $t>0$,  and moreover 
\[
\lim_{t\to 0} \bigl \|    \widehat W_t (Q,P)  d_{P,w _t}     -   d_{Q,w_t}  \widehat W_t(Q,P) \bigr \| = 0.
\]
\end{proposition}

Recall now that the Julg-Valette differential is defined using the operation $H\wedge C$ between hyperplanes and cubes.  Since the operation depends on a choice of base vertex, we shall from now on write $H\wedge_P C$ to indicate that choice, as we did earlier. 

To prove   Proposition~\ref{prop-eq-technical2} it suffices to consider the case where  $P$ and $Q$ are at distance $1$ from one another (so they are separated by a unique hyperplane).  We shall make this assumption from now on.

\begin{lemma}
\label{lemma-10-tech1}
If a hyperplane $H$ fails to  separate $P$ from $Q$, then
\[
H \wedge_ P \widehat W_t (P,Q)D  = \widehat W_t (P,Q) ( H\wedge_ Q   D )  
\]
for all oriented $q$-cubes $D$.
\end{lemma}

\begin{proof}
First, if $H$ fails to separate $P$ from $Q$, then the operators $H\wedge _P $ and $H\wedge _Q$ are equal to one another.  We shall drop the subscripts for the rest of the proof.

Next, if $H$ cuts $D$, then it cuts all the cubes parallel to $D$, and therefore it cuts all    the cubes that  make up $\widehat W_t (P,Q)D$.  So both sides of the equation in the lemma are zero.  So can assume from now on that $H$ is disjoint from $D$.

Let  $K$ be the hyperplane that separates $Q$ from $P$. According to   Proposition \ref{prop-nearest-cube2} the nearest $q$-cubes to $P$ and $Q$  in the parallelism class of $D$ are either equal or are opposite faces, across $K$,  of a $(q{+}1)$-cube that is cut by $K$.  So $\widehat W_t (P,Q)D$ is either just $D$ or is a combination 
\begin{equation}
\label{eq-w-hat-of-d}
\widehat W_t (P,Q)D = a D + b E
\end{equation}
of $D$ and another cube $E$ that is an opposite face  from $D$ in a $(q{+}1)$-cube that is cut by $K$.  

We see that if $H$ fails to separate $D$ from $P$, or equivalently, if it fails to separate $D$ from $Q$, then it also fails to separate any of the terms in $\widehat W_t (P,Q)D$ from $P$ or $Q$, and accordingly  both sides of the equation in the lemma are zero.  So we can assume from now on that $H$ does separate $D$ from $P$ and $Q$.
  
Suppose now that $K$ fails to be adjacent to $D$, either because it cuts $D$ or because some vertex of $D$ is not adjacent to $K$.   The left-hand side of the equation is then $H\wedge D$.  This is either zero, in which case the equation obviously holds, or it is a $(q{+}1)$-cube to which $K$ also fails to be adjacent, in which case the right-hand side of the equation is simple $H\wedge D$.  So we can assume that $K$ is adjacent to $D$.

Let $E$ be the $q$-cube that is separated from $D$ by $K$ alone, as in  \eqref{eq-w-hat-of-d}. Since $H$ fails to separate $D$ from $E$, or $P$ from $Q$, but separates $D$ and $E$ from $P$ and $Q$, we see from Lemma~\ref{lem-4-quadrants} that $H$ and $K$ intersect.   By Lemma~\ref{lem-k-hplanes}, if $H$ is adjacent to either of $D$ or $E$, then there is a $(q{+}2)$-cube that is cut by $H$ and $K$ and contains both $D$ and $E$ as faces.  In this case both sides of the equation in the lemma are 
\[
a\, H\wedge D\, + \,b\, H \wedge E
\]
with $a$ and $b$ as in \eqref{eq-w-hat-of-d}.  Finally, if $H$ is adjacent to neither $D$ nor $E$, then both sides of the equation are zero.\end{proof}

\begin{lemma}
\label{lemma-10-tech2}
If $H$ separates $P$ from $Q$, then 
\[
H \wedge_ P \widehat W_t (P,Q)D -  \widehat W_t (P,Q) ( H\wedge_ Q   D )  =    f(t) H\wedge _Q D - g(t) H\wedge _P D ,
\]
where $f$ and $g$ are smooth, bounded  functions on $[0,\infty)$  that vanish at $t=0$. \end{lemma}

\begin{proof}
If $D$ fails to be adjacent to $H$, then both sides in the displayed formula are zero.  So suppose $D$ is adjacent to $H$.  In this case 
\[
\widehat  W_t (P,Q) ( H\wedge_ Q   D ) = H\wedge_Q D .
 \]
 Now according to the definitions 
 \[
 \widehat W_t (P,Q)D  = \pm e^{-\frac 12 t^2} E + (1-e^{- t^2})^{\frac 12} D
 \]
 where $E$ is the $q$-cube opposite $D$ across $H$, and where the sign is $+1$ if  $D$ is separated from $P$ by $H$, and $-1$ if it is not.  We find then that 
 \[
 H \wedge_ P \widehat W_t (P,Q)D = \pm e^{-\frac 12 t^2}  H\wedge_P E + (1-e^{- t^2})^{\frac 12} H\wedge _P D.
 \]
 But  $H \wedge _P E = 0$ if  $E$ is not separated from $P$ by $H$, which is to say if $D$ \emph{is} separated from $P$ by $H$.  So  we can write 
 \[
  H \wedge_ P \widehat W_t (P,Q)D =- e^{-\frac 12 t^2}  H\wedge_P E + (1-e^{- t^2})^{\frac 12} H\wedge _P D.
 \]
 In addition
 \[
 H\wedge_P E = - H\wedge _Q D 
 \]
 so that 
 \[
   H \wedge_ P \widehat W_t (P,Q)D = e^{-\frac 12 t^2}  H\wedge_Q D + (1-e^{- t^2})^{\frac 12} H\wedge _P D.
 \]
 Finally we obtain
 \[
  \widehat W_t (P,Q) ( H\wedge_ Q   D )  -   H \wedge_ P \widehat W_t (P,Q)D  
  =
  (e^{\frac 12 t^2}-1) H\wedge _Q D - (1-e^{- t^2})^{\frac 12} H\wedge _P D,
 \]
 as required.
\end{proof}

\begin{proof}[Proof of Proposition~\ref{prop-eq-technical2}]
We shall use the previous lemmas and the formula
\[
d_{P,w_{t}} D = \sum _H w_{P,t}(H) \, H \wedge _P D ,
\]
for the Julg-Valette differential.  We get 
\begin{multline}
\label{eq-p-and-q-jv-maps}
   \widehat W_t (Q,P) d_{P,w_t}    -   d_{Q,w_t}  \widehat W_t(Q,P) \\
   =
   \sum _H  \Bigl ( w_{P,t} (H) \widehat W_t (Q,P)  \left ( H \wedge _P D \right )    -   w_{Q,t}(H) H \wedge _Q     \widehat W_t(Q,P) D  
   \Bigr ) .
\end{multline}
Let us separate the sum into  a part  indexed by hyperplanes that do \emph{not} separate $P$ from $Q$, followed by the single term  indexed by the hyperplane $H_0$ that does separate $P$ from $Q$.  According to Lemma~\ref{lemma-10-tech1} the first part   is 
\begin{equation*}
\sum _{H\ne H_0}   \bigl ( w_{P,t} (H)  -  w_{Q,t}(H)  \bigr )\,  \widehat W_t (Q,P)  \left ( H \wedge _P D \right )  .
\end{equation*}
Inserting the definition of the weight function,  we obtain
\begin{equation}
\label{eq-H-not-H_0}
 t  \sum _{H\ne H_0}   \bigl ( \operatorname{dist}(H, P)  -  \operatorname{dist}(H,Q) \bigr  )\, \widehat  W_t (Q,P)  \left ( H \wedge _P D \right )  ,
\end{equation}
and moreover
\begin{equation*}
\bigl |  \operatorname{dist}(H, P)  -  \operatorname{dist}(H,Q)  \bigr | \le 1.
\end{equation*}
As for the part of \eqref{eq-p-and-q-jv-maps} indexed by $H$, keeping in mind that 
\[
 \operatorname{dist}(H_0, P)  = \tfrac 12  =  \operatorname{dist}(H_0, Q),
 \]
we obtain from Lemma~\ref{lemma-10-tech2} the following formula for it:
\begin{equation}
\label{eq-H-is-H_0}
 (1 + \tfrac 1 2 t ) f(t) H_0\wedge _Q D -  (1 + \tfrac 1 2 t) g(t) H_0\wedge _P D,
\end{equation}
where $f$ and $g$ are bounded and vanish at $0$.  The required estimates follow, because   the terms in  \eqref{eq-H-not-H_0}  and \eqref{eq-H-is-H_0} are uniformly bounded in number, are supported uniformly close to $D$, are uniformly bounded in size, and vanish at $t=0$.
\end{proof}


\bibliographystyle{alpha}
\bibliography{References.bib}


\end{document}